\documentclass{amsart}
\usepackage[utf8x]{inputenc}
\usepackage{amsmath,amsthm}
\usepackage{amsfonts,amstext,amssymb, mathtools, comment, mathtools}
\usepackage{graphicx,color,enumerate}

\RequirePackage[OT1]{fontenc}
\RequirePackage{amsthm,amsmath,float}
\RequirePackage[numbers]{natbib}
\RequirePackage[colorlinks,citecolor=blue,urlcolor=blue]{hyperref}

\newcommand{\p}{{\mathbb P}}
\newcommand{\e}{{\mathbb E}}

\newcommand{\Em}{\e_\mu}
\newcommand{\D}{{\mathrm d}}

\newcommand{\R}{{\mathbb R}}
\newcommand{\1}[1]{\mbox{\rm\large  1}_{\{#1\}}}
\renewcommand{\a}{{\alpha}}

\newcommand{\bs}{\boldsymbol}
\newcommand{\vaguely}{\xrightarrow{\scriptstyle v}} 
\newcommand{\weakly}{\xrightarrow{\scriptstyle w}} 

\DeclareMathOperator\VaR{VaR} 
\DeclareMathOperator\var{var} 
\DeclareMathOperator\cov{cov} 
\DeclareMathOperator\corr{corr}
\DeclareMathOperator\Leb{Leb}

\newtheorem{theorem}{Theorem}
\newtheorem{corollary}{Corollary}
\newtheorem{lemma}{Lemma}
\newtheorem{prop}{Proposition}
\newtheorem{remark}{Remark}

\newtheorem{example}{Example}

\begin{document}
\title{Robust bounds in multivariate extremes}

\author[S.\ Engelke]{Sebastian Engelke}
\address{Ecole Polytechnique F\'ed\'erale de Lausanne\\
   EPFL-FSB-MATHAA-STAT\\
   Station 8\\
   1015 Lausanne\\
   Switzerland\\
   sebastian.engelke@epfl.ch}

\author[J.\ Ivanovs]{Jevgenijs Ivanovs}
\address{Aarhus University\\
  Department of Mathematics\\
  Ny Munkegade\\
  DK-8000 Aarhus C\\
  Denmark\\
 jevgenijs.ivanovs@math.au.dk}
\begin{abstract}
Extreme value theory provides an asymptotically justified framework for estimation of exceedance probabilities in regions where few or no observations are available. For multivariate tail estimation, the strength of extremal dependence is crucial and it is typically modeled by a parametric family of spectral distributions. In this work we provide asymptotic bounds on exceedance probabilities that are robust against misspecification of the extremal dependence model. They arise from optimizing the statistic of interest over all dependence models within some neighborhood of the reference model. A certain relaxation of these bounds yields surprisingly simple and explicit expressions, which we propose to use in applications. We show the effectiveness of the robust approach compared to classical confidence bounds when the model is misspecified. The results are further applied to quantify the effect of model uncertainty on the Value-at-Risk of a financial portfolio.
%Our theory also provides more general robust bounds under moment constraints that may be of interest in various other contexts.
\end{abstract}

\subjclass[2010]{Primary 60G70, 62G32, 62G35}
\keywords{Extremal dependence, Pickands' function, model misspecification,
stress test, robust bounds, convex optimization}

\maketitle
\section{Introduction}

In parametric statistics there are several sorts of uncertainties that arise
in the estimation of an unknown quantity of interest.
The estimation uncertainty, for instance, refers to the error made by inferring 
the model parameters from only finitely many data points. Bootstrapping or results
on asymptotic normality are typically applied to quantify this error and 
to derive confidence intervals.
On the other hand, the parametric family used as a model for the data is a finitely
dimensional subset of all distributions and is thus only
an approximation of the true data generating distribution. The 
uncertainty due to this misspecification is usually called 
model uncertainty, and it is more difficult to quantify than the
estimation uncertainty within a parametric model class.
A popular way to provide confidence bounds, that are robust against wrong model
assumptions, is to find the smallest
and largest values of the statistic of interest with respect to all
probability measures in some neighborhood of the 
estimated parametric distribution assuming that it contains the true data generating distribution.
Moreover, one may view such a search for the worst case as a systematic stress test within a set of plausible scenarios~\cite{breuer_stress_test}.

For a random vector $(X,\bs Y^\top)=(X,Y_1,\ldots,Y_{d-1})$ with $d\geq 2$, in this paper 
we consider the optimization problem
\begin{align}\label{gen_prob}
  V_\mu(\delta)&=\sup_{\p'}\{\e' X:D_\mu(\p',\p)\leq\delta,\e'\bs Y=\e \bs Y\}, \quad \delta > 0,
\end{align}
where the supremum is taken over all probability measures in the $\delta$-neighborhood
of the reference model $\p$ under the constraint that the expectation of $\bs Y$ is preserved.
Here and in the sequel $\e'$ denotes the expectation under the model~$\p'$, and all the measures are defined on a common measurable space $(\Omega,\mathcal F)$. The proximity $D_\mu(\p',\p)$ will be measured 
in terms of the $L_{\mu}^2$-distance between the densities of $\p'$ and $\p$ with respect to some dominating probability measure~$\mu$, which provides additional flexibility in selection of the neighborhood; it will be shown that the choice $\mu=\p$ essentially results in R\'enyi divergence of order~2. 
%A particularly simple explicit form of the
%optimal value $V(\delta)$ is derived for second order power divergence.
The random variable $X$ is the statistic of interest and the constraint on the expectation of $\bs Y$ allows to incorporate necessary model restrictions. %that are required for all models in the optimization domain. 
They arise naturally in the application of the 
results to estimation of multivariate tail probabilities. 
%An analogous minimization problem with the supremum in \eqref{gen_prob} replaced by an infimum will be solved to compute robust lower bounds.

Importantly, the optimizing $\mu$-density has an appealing form yielding the surprisingly simple, explicit expression
\begin{align}\label{sq-rt}
  \e X + \sqrt{\delta\frac{\det\{ \Sigma_\mu(X,\bs Y)\}}{\det\{\Sigma_\mu(\bs Y)\}}}
\end{align}
for the optimal value $V_\mu(\delta)$ when $\delta\in[0,\delta^*]$ is in a certain range, and otherwise this expression provides an upper bound on $V_\mu(\delta)$, where $\Sigma_\mu(\cdot)$ denotes the respective $\mu$-covariance matrix.
 In this paper, we advocate using this simple square-root bound, and its analogue for the respective minimization problem, as robust bounds for $\e X$ under moment constraints; see Theorem~\ref{thm:main}.
Interestingly, the above fraction of the determinants is a well-known expression in stochastic simulation theory where it arises as the minimal variance of $X+\bs c^\top\bs Y,$ for arbitrary $\bs c\in\mathbb R^{d-1}$~\cite[Sec.\ V.2]{simulation}.

The general optimization problem \eqref{gen_prob} might be interesting in many
different situations, see e.g.~\cite{dey_juneja, glasserman2014robust, breuer_stress_test} for applications of the robust approach to various problems in economics, risk and finance. Let us also note that a problem similar to~\eqref{gen_prob} appears as the dual representation of a coherent risk measure~\cite{ahmadi,delbaen}.
In this work we concentrate on the application to
the risk of rare events and the estimation of their
small tail probabilities, a field that has attracted strong attention in the last decade.
Extreme value theory provides 
%asymptotic results for the convergence
%of normalized maxima of random variables and is thus the 
the theoretical foundation
for statistical extrapolation into tail regions with few or no data;
see \cite{emb1997, deh2006a, res2008} for more details.

The univariate theory is well understood and is concerned with the quantification
of tail probabilities $\p(Z > z)$ of a random variable $Z$, where $z>0$ is a threshold
close to the upper end point of its distribution function~$F$.
%The extremal types theorem states that all possible limits
%of affinely normalized maxima of independent copies of~$Z$ are given by the parametric family of generalized extreme value distributions \cite{fis1928,gne1943}, therefore providing a mathematically sound way to extrapolate $\p(Z>z)$ out of the sample.
There are standard procedures to build confidence intervals for estimators of $\p(Z>z)$,
but bounds that are robust against violation of the assumptions of the extremal types theorem
have only recently been studied in \cite{blanchet2016extreme}. The authors of this paper
solve the optimization problem
$$ \overline F_{\delta}(z) = \sup_{\p'} \{\p' (Z > z): \widehat D(\p',\p) \leq \delta\},$$
where $\overline F_{\delta}$ is the worst case survival function over all 
probability measures $\p'$ in some divergence neighborhood with radius $\delta>0$ around the reference model~$\p$. Here $\widehat D$ is either the Kullback--Leibler divergence or the R\'enyi divergence of an arbitrary order; see also Section~\ref{appendix_KL}.
It is shown in~\cite{blanchet2016extreme} that the worst case tail $\overline F_{\delta}$
is considerably heavier than the one of the reference distribution $F$.

For a $d$-dimensional random vector $\bs Z = (Z_1,\ldots, Z_d)^\top$, multivariate extreme value theory
studies probabilities $\p( \bs Z \in tB)$, where for $B\subset [0,\infty]^d$ 
bounded away from the origin and large $t>0$ the dilated set $t B$ is called a tail region.
As in the univariate case, the idea is to extrapolate from regions with more
data into the tails, but in the multivariate case the dependence between components $Z_i$
at high quantiles is crucial. 
The mathematical concept of regular variation is
needed in order to perform this extrapolation.
Assuming that $\bs Z$ is standardized to have unit Pareto marginal tails, multivariate extreme value theory justifies, in particular, the following approximation for any $z_i>0$ and large~$t>0$:
\begin{equation}\label{eq:approx_main}
\p(\exists i:Z_i>t z_i)\approx t^{-1} d \,\e \left(\max_{i=1}^d \frac{Y_i}{z_i}\right), 
\end{equation}
where $\bs Y = (Y_1,\dots, Y_d)$ takes values in the standard simplex and satisfies certain moment constraints, i.e.,
\[\bs Y\in\mathbb S^{d-1}=\left\{\bs y\in[0,\infty]^d:\sum_{i=1}^dy_i=1\right\}\qquad\text{ and }\qquad\forall i:\e Y_i=1/d.\]
%If $\bs Z$ has standard Fr\'echet margins and is regularly varying on $E$ with limiting  Radon measure $\nu$, that is,
%\begin{align}\label{reg_var}
%  n\p( \bs Z /n \in \, \cdot\,) \to \nu(\, \cdot\, ), \quad n\to\infty,
%\end{align}
%in the sense of vague convergence, then the tail probability $\p( \bs Z \in n A)$ can be approximated 
%by $\nu(A)/n$; see \cite[Chapter 5]{res2008} for details. The so-called \emph{exponent measure} $\nu$ describes the tail dependence between $Z_1$ and $Z_2$
%and is known to satisfy
%\begin{align}\label{eq:nu_Y}
%  \nu(A_{\bs z}) = 2\e \left(\frac{Y}{z_1}\vee \frac{1-Y}{z_2}\right),\qquad A_{\bs z}=E\backslash [\bs 0,\bs z]
%\end{align}
%for some random variable $Y\in [0,1]$ with $\e Y=1/2$ and all $\bs z\in E$. The right hand side of~\eqref{eq:nu_Y} can be written as $2\theta(z_1/(z_1+z_2))/(z_1+z_2)$, where
%\begin{equation}\label{eq:theta_z}
%\theta(z)=\e\left(\frac{Y}{z}\vee\frac{1-Y}{1-z}\right),\qquad z\in[0,1].
%\end{equation} 
The distribution of $\bs Y$ is called a \emph{spectral distribution} and it encodes extremal dependence in the model. Many parametric models have been proposed for the spectral distribution \cite[e.g.,][]{hue1989, taw1990, bol2007,coo2010}. For a non-parametric approach to estimation of the spectral distribution we refer the reader to~\cite{einmahl_segers}, where an optimization problem is used to enforce the moment constraint.

A natural problem is to find 
bounds for the asymptotic expression of the tail probability in~\eqref{eq:approx_main} with fixed $\bs z =(z_1, \dots, z_d)^\top$ that are robust
against model misspecification of the spectral distribution, i.e., the distribution of~$\bs Y$. For the upper bound we are thus interested in the
maximization problem 
\begin{align}\label{max_prob}
  \sup_{\p'}\left\{\e' \left(\max_{i=1}^d\frac{Y_i}{z_i}\right) :D_\mu(\p',\p)\leq\delta,\e' Y_i= 1 / d \text{ for all }i\right\}
\end{align}
which is clearly a special case of~\eqref{gen_prob} with $X=X(\bs z)=\max_{i=1}^dY_i/z_i$.
Importantly, we assume here that the dominating measure $\mu$ is supported by~$\{\bs Y\in\mathbb S^{d-1}\}$ and hence $\bs Y\in \mathbb S^{d-1}$ holds also~$\p'$-a.s. %, see Section~\ref{sec:measure_space} for further details concerning the choice of the underlying measurable space.
%Alternatively, we may reformulate this optimization problem using the respective distributions of $\bs Y$, that is, we consider the measurable space $\mathbb S^{d-1}$ equipped with its Borel $\sigma$-algebra.} 
In particular, $Y_d=1-\sum_{i=1}^{d-1} Y_i$ and so there are essentially $d-1$ moment constraints which ensure that $\bs Y$ has a valid spectral distribution also under the measure~$\p'$.

 Similarly, a lower bound can be defined as
the optimal value of the corresponding minimization problem with $\sup$ replaced by $\inf$ in \eqref{max_prob}.  The respective optimal values $\beta^*(\bs z)$ and $\beta_*(\bs z)$ of these optimization problems readily yield the robust asymptotic  bounds
\begin{align*} t^{-1} d \,\beta_*\left(\bs z\right)\lesssim \p(\exists i:Z_i>tz_i)\lesssim t^{-1} d \,\beta^*\left(\bs z\right), \qquad \text{as }t\to\infty.
\end{align*}
Note that according to~\eqref{eq:approx_main} it is enough to consider $\bs z\in\mathbb S^{d-1}$.
It should also be stressed that our bounds address misspecification of the extremal dependence model exclusively, and so they are guaranteed to hold for sufficiently large scaling factor~$t$ only. 
%The choice of an appropriate threshold $\delta$ is the Achiless heel of the robust approach.
Furthermore, we essentially optimize over the class of max-stable distributions, which is different from the univariate case analysis in~\cite{blanchet2016extreme}.

%For simplicity, we only consider the bivariate case, but all results apply
%to robust bounds for tail probabilities in higher dimensions $d\geq 2$
%with $n = d-1$ moment constraints in \eqref{gen_prob}.

In Section \ref{sec:model_risk} we provide details on the divergence $D_\mu(\p',\p)$,
and recall necessary results on multivariate
extreme value theory, regular variation and spectral measures.
The convex optimization problem \eqref{gen_prob} is solved in Section \ref{sec:convex} and the simple square-root bound for the optimal value~$V_\mu(\delta)$ is derived in Section~\ref{sec:rob_bounds}, where we also identify a necessary and sufficient condition for this upper bound to coincide with~$V_\mu(\delta)$. Based on these general results, in Section~\ref{sec:pickands_bounds} 
we investigate robust bounds for small
probabilities of tail regions in the bivariate case arising from the optimization problem \eqref{max_prob} for~$d=2$. Several examples are given in Section~\ref{sec:illustration} to illustrate the
results. In Section \ref{sec:exp} we conduct an experiment that 
shows the effectiveness of
the robust bounds compared to classical confidence bounds when the model
is misspecified. As a further application of our theory, Section \ref{sec:var} discusses 
how worst case bounds on the Value-at-Risk of a financial portfolio under model uncertainty
can be derived.
The Appendix contains some parametric families of spectral distributions, further comments about the degenerate maximizer of the problem in~\eqref{max_prob},
and results on optimization for other divergences.
%This material is not strictly necessary for application of the theory and so it is placed at the end.

%\footnote{Can cite \cite{klu2006} for extreme value models or Pickands function; Other possible citation for extremes in finance: Kl\"uppelberg ``Risk management with extreme value theory'', 2004}

\section{Preliminaries and the setup}\label{sec:model_risk}
\subsection{Distribution model risk}\label{sec:divergence}
Distribution model risk refers to the error made when using a simplified model of reality that is only an approximation to the data generating process. From a probabilistic point of view,
this amounts to computing the quantity of interest, say the probability $\p(A)$ of some event $A$, using a wrong probability measure $\p$, which nevertheless is close in some sense to the true measure~$\p_{\textup{true}}$. The robust approach to this problem is to consider all measures~$\p'$ in some neighborhood of~$\p$ that 
should contain $\p_{\textup{true}}$ as well, and to find the maximal and the minimal values among all~$\p'(A)$.
These numbers then provide robust bounds on the true value $\p_{\textup{true}}(A)$.
This approach has become quite popular in financial mathematics, see~\cite{hansen2001robust,ahmadi,breuer2013measuring,glasserman2014robust} and references therein, and~\cite{blanchet2016extreme} for an application to univariate extreme value statistics.

A natural way to define a neighborhood of measures around $\p$ is to consider 
some form of divergence. % between two measures $\p$ and $\p'$. 
Fix a dominating probability measure~$\mu$, i.e., such that $\p\ll\mu$,
and suppose for now that $\p'\ll\mu$. Letting $L=\D\p/\D \mu$ and $L'=\D\p'/\D \mu$ be the corresponding Radon--Nikodym derivatives we consider the standard squared $L^2_\mu$-distance
\begin{equation}\label{eq:divergence_ref}D_\mu(\p',\p)=\e_\mu(L'-L)^2,\end{equation}
where $\e_\mu$ denotes the expectation under probability measure~$\mu$. We put $D_\mu(\p',\p)=\infty$ if $\p'$ is not absolutely continuous with respect to~$\mu$. It is noted that~\eqref{eq:divergence_ref} is a special case of the so-called Bregman divergence, see, e.g.,~\cite{breuer2013measuring}.
Furthermore, by choosing $\mu=\p$ we get
\begin{equation}\label{eq:divergence}D_\p(\p',\p)=\e(L'-1)^2=\e L'^2-1\end{equation}
for all $\p'\ll\p$ with $L'=\D\p'/\D\p$. Moreover, 
\[D_\p(\p',\p)\leq \delta\qquad\text{iff}\qquad \log\e L'^2\leq \log(1+\delta)=\delta',\]
where $\log\e L'^2$ is the well-known R\'enyi (power) divergence of order~$2$ of $\p'$ from~$\p$. In other words, neighborhoods of measures defined by $D_\p(\cdot,\p)\leq \delta$ coincide with second order R\'enyi divergence neighborhoods with radius~$\delta'$.

It is clear that the choice of the dominating measure~$\mu$ has an impact on the solution of the optimization problem~\eqref{gen_prob}.
Suppose, for instance, that $\p,\p'$ and $\mu$ are defined on $[0,1]$ and that they are absolutely continuous with respect to Lebesgue measure with densities $f,f'$ and $g$, respectively.
Then it holds that
\begin{equation}\label{eq:unif}D_\mu(\p',\p)=\int_0^1\left(\frac{f'(\omega)}{g(\omega)}-\frac{f(\omega)}{g(\omega)}\right)^2 g(\omega)\D \omega=\int_0^1(f'(\omega)-f(\omega))^2 \frac{1}{g(\omega)}\D \omega,\end{equation}
and so $\mu$ provides a mechanism of weighing the squared distance between $f'$ and~$f$. A similar weight function appears in e.g.~\cite{bucher_pickands_fun} in the context of estimating the Pickands' function.
Thus the dominating measure $\mu$ may be chosen according to our uncertainty about the measure~$\p$. 
%It seems reasonable to try and choose $\mu$ in such a way that the divergence between (unknown) $\p_0$ and $\p$ is small, so that the robust bounds are not overly wide. At the same time the choice of $\mu$ not only influences $\delta$ but also the optimal value associated to each~$\delta$. 

In this study we leave out a detailed analysis of the choice of~$\mu$. 
Our default choice in applications to multivariate extremes is $\mu=\p$, which corresponds to R\'enyi divergence of order~2. We also provide an example where this choice is inappropriate, in which case the uniform dominating measure is used. Finally, the remaining parameter $\delta>0$, representing our trust into the measure $\p$, has to be chosen by hand or derived from data. In Section~\ref{sec:exp} we use a straightforward heuristic procedure to estimate it from data. %\textcolor{red}{In fact, choosing an appropriate threshold $\delta$ is the Achilles heel of the robust approach.}\footnote{Leave this out?}

\subsection{Regular variation and spectral distributions}\label{sec:reg_var}
%For the sake of simplicity, we will focus throughout the paper
%on bivariate extremes, but all results apply in higher dimensions.
A $d$-dimensional random vector $\bs Z$ is \emph{multivariate regularly varying} in the non-negative orthant if there exists a sequence $a_t\to\infty$, as $t\to \infty$, and a Radon measure $\nu$ on $E=[0,\infty]^d\backslash\{0\}$ equipped with its Borel $\sigma$-algebra such that 
\begin{equation}\label{eq:reg_var}t\p(\bs Z/a_t\in\cdot)\vaguely\nu,\qquad t\to\infty\end{equation}
in the sense of vague convergence, see, e.g.,~\cite[Ch.\ 6]{resnick_heavy}.
The so-called \emph{exponent measure} $\nu$ then satisfies the scaling property
$\nu(t B)=t^{-\a}\nu(B)$
for all $t>0$ and all Borel sets~$B\subset E$ bounded away from~0, where $\a>0$ is called the \emph{tail index} of regular variation. 
Moreover, by switching to polar coordinates $\bs  z\mapsto (\|\bs z\|,\bs z/\|\bs z\|)=(r,\bs \omega)$ for the $L_1$-norm $\|\bs z\|=\sum_{i=1}^d|z_i|$ on $\R^d$, the measure $\nu$ factorizes into
\[c \, \a r^{-{\a-1}}\D r\times H(\D \bs\omega),\]
where $c>0$ and $H$ is a probability measure, called the \emph{spectral measure}, on the simplex $\mathbb S^{d-1}$ equipped with its Borel $\sigma$-algebra. Importantly, \eqref{eq:reg_var} implies the following weak convergence to~$H$:
\begin{equation}\label{eq:limH}\p\left(\left.\frac{\bs Z}{\|\bs Z\|}\in\cdot\, \right|\, \|\bs Z\|>t\right)\weakly H, \qquad \text{as } t\to\infty.
\end{equation}

Without loss of generality we assume that $\nu$ is non-degenerate in the sense that 
$\nu(\{\bs z:z_i>1\})\neq 0$ for all~$i=1,\dots, d$. Otherwise, we may simply remove the components of the vector~$\bs Z$ that decrease at a faster rate. This implies that all marginal survival functions $\overline F_i(z)=1-F_i(z)$ are regularly varying with the same index~$-\a$, and, moreover, for some $m_i>0$, 
\begin{equation}\label{eq:Fequiv}\frac{\overline F_i(z)}{\overline F_1(z)}\to m_i\qquad\text{ as }z\to\infty\end{equation}
with $m_1=1$. That is, $\overline F_i(z)$ are equivalent in the limit up to multiplicative constants. 

It is common to split the problem of multivariate tail estimation into estimation of marginal tails and estimation of the spectral distribution.
The theory for univariate tail estimation is well-studied and there are many
established methods to estimate the survival functions~\cite{deh2006a,resnick_heavy}.
We therefore assume that the marginal tail models are continuous
and correctly specified, and that the $Z_i$ have been transformed to
unit Pareto tails. That is, we generally assume that $\a=1$ and $\overline F_i(z)=1/z$ for large $z$,
apart from Section~\ref{sec:var}, where we return to the general setup and the issue of standardization.
%Furthermore, in practice one rarely observes the same marginal tail indices, and so it is often necessary to standardize the components $Z_i$ before one can proceed with multivariate extreme value theory. 

With the above standardization in mind we may choose $a_t=t$ in~\eqref{eq:reg_var} leading to the approximation
\[\p(\bs Z\in tB)\approx \nu(B)/t\]
for large $t$ and $B$ bounded away from the origin with $\nu(\partial B) = 0$.
A natural choice of such a set is given by $B_{\bs z}=E\backslash [\bs 0,\bs z]$, where we may assume that $\|\bs z\|=1$ because of the scaling property of $\nu$. That is, we are interested in approximating the probability that at least one marginal is relatively large, namely $Z_i>tz_i$ for some~$i$. 
%Importantly, $G(\bs z)=\exp(-\nu(B_{\bs z}))$ defines the corresponding max-stable distribution~\cite[Sec.\ 5.3]{res2008}.
Letting $\bs Y\in\mathbb S^{d-1}$ have the spectral distribution~$H$ one finds that
\begin{equation}\label{eq:nuBz}
\nu(B_{\bs z})=c\e\int \1{\exists i:rY_i>z_i} r^{-2} \D r=c\e \left(\max_{i=1}^d\frac{Y_i}{z_i}\right).
\end{equation}
Moreover, according to the above standardization the exponent measure must satisfy
\[t\p(Z_i>t)\to 1=\nu(\{\bs z:z_i>1\})=c\e Y_i\]
and hence $\e Y_i=1/c$ for all~$i$. But since $\sum_{i=1}^d Y_i=1$, it must be that $c=d$,
which yields the approximation in~\eqref{eq:approx_main}, our starting point for the robust approach.
 Importantly, any $\bs Y$ satisfying these moment constraints gives rise to a valid spectral measure.
 
\begin{remark}\label{rem:norm} 
The sum norm used throughout this paper is special in the sense that the constant $c$ does not depend on the spectral measure. This makes it possible to employ the optimization problem in~\eqref{gen_prob}. Other norms would lead to the objective~$\e' X(\bs z)/\e' Y_1$, which does not comply with~\eqref{gen_prob}.
\end{remark}

A common way of representing the dependence structure in the bivariate case~\cite{pickands,klu2006,bucher_pickands_fun} is by means of the so-called
\emph{Pickands' function}
\begin{equation}\label{eq:pickands}
A(z)=2\e\{(1-z)Y_1\vee z(1-Y_1)\},\qquad z\in[0,1].
\end{equation}
Indeed, an easy transformation of~\eqref{eq:nuBz} yields
\[\nu(B_{\bs z})=\left(\frac{1}{z_1}+\frac{1}{z_2}\right)A\left(\frac{z_1}{z_1+z_2}\right).\]
Importantly, the Pickands' dependence function $A:[0,1]\to [1/2,1]$ is convex and satisfies $z\vee(1-z) \leq A(z) \leq 1$. Moreover, any such function defines a unique exponent measure~$\nu$, see~\cite[p.\ 226]{deh2006a}.

\section{Convex optimization}\label{sec:convex}
In this section we solve the optimization problem~\eqref{gen_prob} which, according to~\eqref{eq:divergence_ref}, can be rewritten in the convenient form
\begin{align}\label{eq:opt}
V_\mu(\delta)
&=\sup_{L'\geq 0}\{\Em(L'X): \Em L'=1,\Em(L'-L)^2\leq \delta,\Em(L'\bs Y)=\Em(L\bs Y)\},
\end{align}
where the supremum is taken over all measurable functions $L':\Omega \to [0,\infty)$ satisfying the stated constraints.
%Similarly to~\eqref{eq:divergence}, we let $V(\delta)=V_\p(\delta)$ which corresponds to R\'enyi divergence of order~2.
This is a convex optimization problem in an infinite dimensional space allowing for a rather explicit solution given in Theorem~\ref{thm:opt}. For related results without moment constraints see~\cite{breuer2013measuring,blanchet2016extreme,dey_juneja,glasserman2014robust}. The latter two works also provide short derivations based on the strong duality theorem. There is, however, no reference to the strong duality theorem for infinite dimensional spaces which does require verification of certain conditions. Moreover, the issue with a distribution of~$X$ with some mass at its right end is not addressed in the literature.

\subsection{The underlying measurable space}\label{sec:measure_space} 
Before solving the optimization problem~\eqref{gen_prob} or its equivalent version~\eqref{eq:opt}, let us comment on the underlying measurable space~$(\Omega,\mathcal F)$. Letting $\mathcal G=\sigma(X,\bs Y)\subset \mathcal F$ we assume that $L=\D \p /\D\mu$ is $\mathcal G$-measurable. That is, the choice of the dominating measure~$\mu$ does not introduce additional randomness in the model, which is trivially the case for our default choice~$\mu=\p$.
The optimization problem~\eqref{gen_prob} formulated on the measurable space $(\Omega,\mathcal F)$ and its analogue formulated on the measurable space $(\Omega,\mathcal G)$ lead to the same optimal value~$V_\mu(\delta)$. This follows from Jensen's inequality:
\[\e_\mu(L'-L)^2=\e_\mu\left(\e_\mu[(L'-L)^2|\mathcal G]\right)\geq \e_\mu(\e[L'|\mathcal G]-L)^2,\]
where the latter is the respective divergence on $(\Omega,\mathcal G)$. 
Therefore, we may always consider the induced distributions of $(X,\bs Y^\top)$ without changing the robust bounds. In the setting of~\eqref{max_prob} we may thus work on the Borel $\sigma$-algebra of $\mathbb S^{d-1}$. In fact, this can be seen as the modeling choice requiring little justification.
%In any case, the choice of the measurable space  underlying~\eqref{gen_prob} is just a modelling assumption.

\subsection{The optimal Radon--Nikodym derivative}
Let us immediately present the solution to the optimization problem~\eqref{eq:opt}. It is noted that the proof of this result provides good intuition on the form of the solution. Throughout the paper, we will denote a maximizer of \eqref{eq:opt}, if it exists, by $L^*$, and for any random variable $X$ we put $\e^* X = \e_\mu (L^* X)$.
\begin{theorem}\label{thm:opt}
Assume that $\Em X^2,\Em Y^2_i,\Em L^2<\infty$ and let $\e \bs Y=\bs y$.
Then $L^*$ is a maximizer of the optimization problem~\eqref{eq:opt} if and only if $\Em L^*=1,\Em (L^*\bs Y)=\bs y$ and at least one of the following holds:
\begin{itemize}
\item[(i)] there exist $a>0,b, c_i\in \R, i=1,\dots, d-1$, such that 
\begin{gather*}L^*=\left(a X+b+\bs c^\top \bs Y+L\right)_+\quad \mu\text{-a.s.\ and }\qquad
\Em (L^*-L)^2=\delta.
\end{gather*}
\item[(ii)] there exist $c_i\in \R, i=1,\dots, d-1$, such that the distribution of $X+\bs c^\top\bs Y$ under $\mu$ has a positive mass at its upper end, $L^*=0$ everywhere else $\mu$-a.s., and the constraint $\Em(L^*-L)^2\leq \delta$ holds.
\end{itemize}
\end{theorem}
\begin{proof} 
Note that $\Em L'^2\leq 2\{\Em (L'-L)^2+\Em L^2\}<\infty$ if $\Em (L'-L)^2\leq \delta$. So we may consider a normed vector space of $\mu$-square-integrable $L'$ and its convex subset defined by the additional requirement of $L'\geq 0$. Note also that $\Em(L'|X|),\Em(L'|Y_i|)<\infty$.
Next, for the convex optimization problem~\eqref{eq:opt} we define the corresponding Lagrangian:
\begin{equation}\label{eq:dual}\mathcal L(L')=\Em(L'X)-a(\Em (L'-L)^2-\delta)+b(\Em L'-1)+\bs c^\top(\Em(L'\bs Y)-\bs y),\end{equation}
where $a\geq 0,b,c_i\in\R$. The strong duality theorem, see e.g.~\cite[Thm.\ 4]{millar}, asserts that $L^*$ is a maximizer of the original problem if and only if $L^*$ is a maximizer of $\sup_{L'\geq 0}\mathcal L(L')$ for some $a\geq 0,b,c_i\in\R$, such that the constraints hold as well as so-called complementary slackness: 
\[\Em L^*=1,\quad \Em(L^*\bs Y)=\bs y,\quad \Em (L^*-L)^2\leq \delta,\quad a(\Em (L^*-L)^2-\delta)=0.\] For this result to be true it is sufficient to verify Slater's condition: $\exists L'\geq 0$ such that $\Em (L'-L)^2<\delta$ and $\Em L'=1,\Em(L' \bs Y)=\bs y$, but this is clearly satisfied by $L'=L$.

Hence it is left to solve the dual problem $\sup_{L'\geq 0}\mathcal L(L')$ for fixed $a\geq 0,b,c_i\in\R$. Since $\mathcal L(L')$ is concave in $L'$, a sufficient and necessary condition for a maximizer $L^*$ of the dual problem is 
\[g_{L^*,L'}'(0+)\leq 0\quad \forall L'\geq 0, \text{ where}\qquad g_{L^*,L'}(t)=\mathcal L(L^*(1-t)+L't),\]
that is, one looks down from~$L^*$.
But $g_{L^*,L}(t)$ is given by
\[\Em\left\{(L^*(1-t)+L't)(X+b+\bs c^\top\bs Y)-a(L^*(1-t)+L't-L)^2+a\delta-b-\bs c^\top\bs y\right\},\]
which can be differentiated under the expectation sign, see e.g.~\cite[A16]{williams}, yielding
\begin{equation}\label{eq:ineq}\Em(L'-L^*)(X+b+\bs c^\top\bs Y-2a(L^*-L))\leq 0,\quad\forall L'\geq 0.\end{equation}
This implies that $X+b+\bs c^\top \bs Y-2a(L^*-L)\leq 0$ and $L^*=0$ when the inequality is strict $\mu$-a.s., because otherwise we may choose $L'\geq 0$ to invalidate~\eqref{eq:ineq}. But the latter clearly implies~\eqref{eq:ineq} and so we have the equivalence.  Thus for $a>0$ we get
\[L^*=\left(\frac{X+b+\bs c^\top\bs Y}{2a}+L\right)_+\quad \mu\text{-a.s.},\]
which is equivalent to~(i). If $a=0$ then $X+\bs c^\top \bs Y\leq -b$ and $L^*=0$ when the inequality is strict $\mu$-a.s. Hence $\mu(X+\bs c^\top\bs Y=-b)>0$, because otherwise $\Em L^*=0$. This yields~(ii).
\end{proof}

\begin{remark}\label{rem:case_ii}
Suppose that for some $\delta'>0$ there is $L^*$ as in (ii) of Theorem~\ref{thm:opt} which also satisfies the equality constraints. Then such $L^*$ must be a maximizer of~\eqref{eq:opt} for any $\delta\geq \delta'$. This implies that the corresponding optimal value $V_\mu(\delta')$ is the maximal possible for any~$\delta> 0$, and in particular it does not increase with further increase of~$\delta$.
In the following we let $\delta^{**}$ be the minimal such $\delta'$, and $\delta^{**}=\infty$ if no such $\delta'$ exists. The optimizer $L^*$ then has the form given in (ii) of Theorem~\ref{thm:opt} if and only if $\delta\geq \delta^{**}$.
\end{remark}

We believe that some further clarification of Theorem~\ref{thm:opt} is necessary. Normally, we only need to look at (i), whereas (ii) corresponds to a rather pathological case explained in Remark~\ref{rem:case_ii}. A necessary condition for the latter is that $\delta$ is sufficiently large, $\delta\geq \delta^{**}$, and also that
$(X,\bs Y^\top)$ satisfies the condition mentioned in (ii) for some~$\bs c$, because otherwise $\delta^{**}=\infty$.
The following two simple examples will provide some further intuition. 

\begin{example}
Consider the optimization problem without moment constraints when $\mu=\p$ and the distribution of $X$ has a positive mass $p>0$ at its upper end~$\overline x$. Then the optimizer $L^*$ in (ii) puts all the mass on $\{X=\overline x\}$ achieving $\e^*X=\overline x$ which is the maximal possible value for any $\delta>0$. But we must have
\[\delta+1\geq \e {L^*}^2=p\e({L^*}^2|X=\overline x)\geq p\e^2 (L^*|X=\overline x)=1/p,\]
because $\e(L^*|X=\overline x)=1/p$.
So if $\delta\geq 1/p-1$ then we can choose $L^*=1/p\1{X=\overline x}$ yielding the maximal possible optimal value $\overline x$, but otherwise we must consider $L^*$ from~(i). In particular, we have $\delta^{**}=1/p-1$.
\end{example}
\begin{example} This example shows that in general (ii) does not require the distribution of $X$ to have a mass at its upper end. Take $\mu=\p$ and consider the case of one constraint where $X=Y1_{A}$ for some event $A$ and $Y>0$ on~$A^c$. Hence $X-Y=-Y1_{A^c}\leq 0$ a.s.\ meaning that $X-Y$ has mass $\p(A)$ at~0.  Clearly, $\e' X\leq \e' Y=y$ for any $\p'\ll\p$, whereas $\e^* X=\e^* Y=y$ if $L^*$ puts all the mass on $A$. It is only left to ensure that there is such $L^*$ preserving the expectation of~$Y$. 
\end{example}

The following observation will later lead to the square-root bound \eqref{sq-rt}, an upper bound on $V_\mu(\delta)$.
\begin{remark}\label{rem:nopositivity}
If the non-negativity constraint on~$L'$ is removed in the optimization problem~\eqref{eq:opt} then 
\[\hat L=aX+b+\bs c^\top \bs Y+L, \qquad a>0,b,c_i\in\R\]
satisfying $\Em \hat L=1,\Em (\hat L-L)^2=\delta,\Em (\hat L\bs Y)=\bs y$ is a maximizer. 
%Moreover, a maximizer may have a different form only in the trivial case when $X+\mu+\bs\eta^\top\bs Y=0$ a.s.\ for some $\mu,\eta_i\in\R$, ans so $\e_L X=-\mu-\bs\eta^\top\bs y$ for any~$L$. 
\end{remark}

Finally, we note that the optimization problem~\eqref{gen_prob} can be solved for some other popular divergences such as R\'enyi and Kullback--Leibler divergences; see Appendix~\ref{appendix_KL} for details. In both cases it is assumed that the dominating measure~$\mu$ coincides with~$\p$. In this paper, however, we aim at providing simple explicit bounds while giving flexibility in defining the neighborhoods of measures by choosing an appropriate dominating measure~$\mu$, and therefore we exclusively focus on the divergence $D_\mu(\p',\p)$ defined in~\eqref{eq:divergence_ref}.

\subsection{Computing the optimal value}\label{sec:computation}
Let us consider the main case (i) of Theorem~\ref{thm:opt} where $L^*=(a X+b+\bs c^\top\bs Y+L)_+$. In order to find $a>0,b,c_i\in\R$, $i=1,\dots, d-1,$ we need to solve a system of $d+1$ corresponding equations:
\begin{equation}\label{eq:3eqns}\Em(L^*-L)=0,\quad \Em(L^*-L)^2=\delta, \quad \Em\{(L^*-L)\bs Y\}=\bs 0.\end{equation}
If a solution is found then the optimal value is given by~$V_\mu(\delta)=\Em(L^*X)$. Note, however, that the maximizer~$L^*$ may be of a different form given in case (ii) of Theorem~\ref{thm:opt}, which corresponds to $\delta\geq \delta^{**}$ and the maximal possible optimal value. In some cases $\delta^{**}$ has an explicit formula, whereas in some other cases identification of $\delta^{**}$ requires solving yet another convex optimization problem. This latter problem usually can be avoided in practice when plotting $V_\mu(\delta)$ as a function of~$\delta$, because solving~\eqref{eq:3eqns} becomes problematic only when $V_\mu(\delta)$ is close to its maximal value. Some further details concerning $\delta^{**}$ in the particular case of problem~\eqref{max_prob} are given in Section~\ref{sec:delta_starstar}.

In the important case of $\mu=\p$ we have $L\equiv 1$. Incorporating the latter into the constant
$b$ reduces the number of equations by one. Indeed, we then may consider $L^*=aU_+$ with $U=U(b,\bs c)=X+b+\bs c^\top\bs Y$ and the equations
\[\e (U_+)=1/a,\quad\var (U_+)=\delta/a^2,\quad\cov(U_+,\bs Y)=\bs 0^\top.\]
Hence we need to find constants $b,c_i\in\R,i=1,\ldots,d-1$, such that $U>0$ with positive probability and
\[\cov(U_+,\bs Y)=\bs 0^\top,\qquad \delta=\var(U_+)/\e^2(U_+).\]
If a solution is found then $V_\p(\delta)=\e(U_+X)/\e(U_+)$ and the corresponding Radon--Nikodym derivative is $L^*=U_+/\e(U_+)$. Moreover, this suggests a parametric approach to plot $V_\p(\delta),\delta>0$: (1) fix $b$ in some range and try to find $c_i, i=1,\dots, d-1,$ such that $\cov(U_+,\bs Y)=\bs 0^\top$ and $U>0$ with positive probability, (2) plot $(\delta,V_\p(\delta))$ for various values of~$b$. 

Solving the above systems of non-linear equations may not be trivial, but it can be done numerically for a moderate dimension~$d$. 
Note that evaluation of the left hand sides in~\eqref{eq:3eqns} for a given choice of constants $a,b,c_i$ requires integration with respect to the joint distribution of $(X,\bs Y^\top)$. Thus each evaluation is costly even for small~$d$.
In the following section we provide an upper bound for the optimal value of a simple explicit form that does not require solving any equation. Moreover, in our applications we observed that this bound often coincides with the optimal value or is very close to it.

\section{Robust bounds of a simple form}\label{sec:rob_bounds}
Throughout this section we assume that $\Em X^2,\Em Y_i^2,\Em L^2<\infty$, and recall that $\Sigma_\mu(X,\bs Y)$ denotes the $\mu$-covariance matrix of the vector~$(X,\bs Y^\top)$. The following result provides a robust bound of a simple form on $\e' X$ under the moment constraints $\e' \bs Y=\e\bs Y$ and in the neighborhood defined by $D_\mu(\p',\p)\leq \delta$. In the following we call it a \emph{robust square-root bound}.
\begin{theorem}\label{thm:main}
%Let $\Sigma$ be a covariance matrix of $(X,\bs Y^\top)$ under $\mu$, and let $\Sigma_\mu(\bs Y)$ be its restriction to the $\bs Y$ variables.
If $\Sigma_\mu(X,\bs Y)$ is invertible then the optimal value of~\eqref{eq:opt} for $\delta > 0$ satisfies
\[V_\mu(\delta)\leq \e X+\sqrt{\delta\frac{\det\{\Sigma_\mu(X,\bs Y)\}}{\det\{\Sigma_\mu(\bs Y)\}}},\]
which {\emph{holds with equality}} if and only if
\begin{align}\label{eq:condition}X-\Em X-\cov_\mu(X,\bs Y)\Sigma_\mu(\bs Y)^{-1}(\bs Y-\Em\bs Y)+\sqrt{\frac{\det\{\Sigma_\mu(X,\bs Y)\}}{\delta\det\{\Sigma_\mu(\bs Y)\}}}L\geq 0\quad  \mu\text{-a.s.}
\end{align}
\end{theorem}
\begin{proof}
The covariance matrix $\Sigma_\mu(X,\bs Y)$ is positive definite and so must be $\Sigma_\mu(\bs Y)$, showing that~$\det\{\Sigma_\mu(X,\bs Y)\},\det\{\Sigma_\mu(\bs Y)\}>0$.
Consider case (i) of Theorem~\ref{thm:opt} and note that we may rescale the constants $b,\bs c$ so that $L^*=(aU+L)_+$, where $U=X+b+\bs c^\top\bs Y$. 

First, we assume that $aU+L\geq 0$ $\mu$-a.s. Then according to~\eqref{eq:3eqns} we have
\[\Em U=0,\qquad \var_\mu(U)=\delta/a^2,\qquad \cov_\mu(\bs Y,U)=\bs 0.\]
Denoting $\bs\sigma=\cov_\mu(\bs Y,X)$ the latter reads as 
\[\bs 0=\cov_\mu(\bs Y,U)=\cov_\mu(\bs Y,X+b+\bs c^\top\bs Y)=\bs \sigma+\Sigma_\mu(\bs Y)\bs c\]
showing that $\bs c^\top=-\bs\sigma^\top\Sigma_\mu(\bs Y)^{-1}$.
Similarly, 
\begin{align}
\var_\mu(U)&=\cov_\mu(U,X+b+\bs c^\top\bs Y)=\cov_\mu(U,X)=\var_\mu(X)+\bs c^\top\bs\sigma\nonumber\\
&=\var_\mu(X)-\bs \sigma^\top\Sigma_\mu(\bs Y)^{-1}\bs\sigma\label{eq:schur}
\end{align}
which is $\det\{\Sigma_\mu(X,\bs Y)\}/\det\{\Sigma_\mu(\bs Y)\}$ according to the well-known formula for the determinant of a block matrix, see~\cite{schur} or~\cite[Eq.\ (1.3)]{schur_in_stat}; the expression in~\eqref{eq:schur} is called the Schur complement of $\Sigma_\mu(\bs Y)$ with respect to~$\Sigma_\mu(X,\bs Y)$. 
Hence we find that 
\[a=\sqrt{\delta\det\{\Sigma_\mu(\bs Y)\}/\det\{\Sigma_\mu(X,\bs Y)\}},\qquad
b=-\Em X+\bs\sigma^\top\Sigma_\mu(\bs Y)^{-1}\Em\bs Y.\]
According to~\eqref{eq:schur} we finally get
\[V_\mu(\delta)=\e^*X=\Em\{(aU+L)X\}=\e X+a\cov_\mu(U,X)=\e X+a\frac{\det\{\Sigma_\mu(X,\bs Y)\}}{ \det\{\Sigma_\mu(\bs Y)\}},\]
which readily yields $V_\mu(\delta)$ under the assumption of non-negativity of~$aU+L$. But in any case we have an upper bound according to Remark~\ref{rem:nopositivity}.
Finally, we have an exact expression for $V_\mu(\delta)$ if
\[0\leq U+L/a=X-\Em X+\bs\sigma^\top\Sigma_\mu(\bs Y)^{-1}\Em\bs Y-\bs\sigma^\top\Sigma_\mu(\bs Y)^{-1}\bs Y+\sqrt{\frac{\det\{\Sigma_\mu(X,\bs Y)\}}{\delta\det\{\Sigma_\mu(\bs Y)\}}}L\]
holds $\mu$-a.s., which completes the proof.
\end{proof}

The assumption that $\Sigma_\mu(X,\bs Y)$ is invertible is not a restriction, because otherwise either some moment constraints are redundant and so can be removed, or $X$ can be expressed as a linear combination of $Y_i$ and so $\e X$ is determined by the moment constraints.
Moreover, there is a link to the control variates method for variance reduction, where $\det\{\Sigma_\mu(X,\bs Y)\}/\det\{\Sigma_\mu(\bs Y)\}$ corresponds to the minimal possible variance of $X+\bs c^\top\bs Y$ for an arbitrary vector~$\bs c$, see~\cite[Sec.\ V.2]{simulation}.  Furthermore, if $(X,\bs Y^\top)$ is jointly normal then the above fraction of the determinants is the variance of $X$ conditional on $\bs Y=\bs y$.
Finally, in some applications it may be important to understand when the optimizing measure is equivalent to the dominating measure~$\mu$, i.e.,~$L^*$ is strictly positive $\mu$-a.s. It is easy to see that this happens if and only if~\eqref{eq:condition} holds with strict inequality. 
As before here we assume that $\det\{\Sigma_\mu(X,\bs Y)\}\neq 0$, which additionally ensures that the case (ii) of Theorem~\ref{thm:opt} cannot be used to construct a strictly positive~$L^*$.

The condition~\eqref{eq:condition} implies that there exists $\delta^*\in[0,\infty]$ such that the robust square-root bound is exact for all $\delta\in[0,\delta^*]$, but otherwise it is conservative.
Because of the form of the square-root bound we observe that necessarily $\delta^*\leq \delta^{**}$, where the latter is defined in Remark~\ref{rem:case_ii}. Figure~\ref{fig:bounds} illustrates these quantities, the optimal value and the square-root bound. Note also that the optimal value~$V_\mu(\delta)$ must be a concave function of~$\delta$, which is easily seen from~\eqref{eq:opt}.
\begin{figure}[h!]
\includegraphics[width=0.3\textwidth]{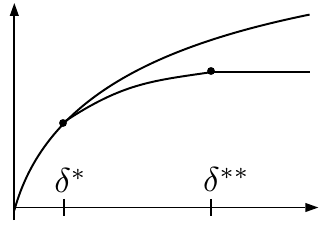}
\caption{The optimal value $V_\mu(\delta)$ (lower curve) and the square-root bound (upper curve) as functions of $\delta$.}
\label{fig:bounds}
\end{figure}
%\bigskip

The corresponding minimization problem is solved by considering $-X$ instead of $X$ in~\eqref{eq:opt}, in which case we define $L_*,\delta_*$ and $\delta_{**}$ analogously to $L^*,\delta^*$ and $\delta^{**}$. In particular, we have the following lower square-root bound.
\begin{corollary}\label{cor:main}
Under the assumptions of Theorem~\ref{thm:main} it holds for $\delta> 0$ that 
\[\inf_{\p'}\,\{\e'X:D_\mu(\p',\p)\leq\delta,\e'\bs Y=\e\bs Y\}\geq \e X-\sqrt{\delta\frac{\det\{\Sigma_\mu(X,\bs Y)\}}{\det\{\Sigma_\mu(\bs Y)\}}}\]
with equality if and only if
\begin{align}\label{eq:condition_min}X-\Em X-cov_\mu(X,\bs Y)\Sigma_\mu(\bs Y)^{-1}(\bs Y-\Em\bs Y)-\sqrt{\frac{\det\{\Sigma_\mu(X,\bs Y)\}}{\delta\det\{\Sigma_\mu(\bs Y)\}}}L\leq 0\qquad  \mu\text{-a.s.}
\end{align}
\end{corollary}
\begin{proof}
Consider $-X$ in place of $X$ in Theorem~\ref{thm:main} and note that $\det\{\Sigma_\mu(X,\bs Y)\}$ stays the same. The condition for equality immediately follows from~\eqref{eq:condition}.
\end{proof}

\begin{remark}\label{rem:mu=p}
When $\mu=\p$, with $\Sigma = \Sigma_\p$ it holds that
\begin{align}\label{eq:delta_star}
\delta^*=\frac{\det\{\Sigma(X,\bs Y)\}}{{b^*}^2\det\{\Sigma(\bs Y)\}}\1{b^*>0}+\infty\1{b^*\leq 0},\qquad 
\delta_*=\frac{\det\{\Sigma(X,\bs Y)\}}{{b_*}^2\det\{\Sigma(\bs Y)\}}\1{b_*<0}+\infty\1{b_*\geq 0},
\end{align}
where $b^*$ and $b_*$ are the essential supremum and the essential infimum of 
\[\cov(X,\bs Y)\Sigma(\bs Y)^{-1}(\bs Y-\e\bs Y)-(X-\e X),\] respectively.
In particular, if $|X|$ and all $|Y_i|$ are bounded a.s.\ then necessarily $\delta^*,\delta_*>0$. 
\end{remark}

In the case of no moment constraints, $d=1$, the square-root bounds on $\e' X$ in the respective ball of measures are given by
\[\e X\pm\sqrt{\delta\var_\mu(X)}.\]
%The upper and lower bounds are precise if
%\begin{align*}X-\Em X+\sqrt{\frac{\var_\mu(X)}{\delta}}L\geq 0\quad \mu\text{-a.s.},\qquad
%X-\Em X-\sqrt{\frac{\var_\mu(X)}{\delta}}L\leq 0\quad \mu\text{-a.s.},
%\end{align*} respectively.
In the case of one constraint, $d=2$, we obtain the square-root bounds
\begin{equation}\label{eq:sqrt_1}
\e X\pm\sqrt{\delta\var_\mu(X)(1-\corr^2_\mu(X,Y))},
\end{equation}
and the corresponding $\delta^*,\delta_*$ can be computed from~\eqref{eq:condition} and~\eqref{eq:condition_min}. Notice that the bounds become tighter in presence of a constraint when $X$ and $Y$ are correlated. 

%\begin{remark}\label{rem:constraints}
It is important to note that the exact robust bounds become tighter or stay the same when an extra moment constraint is added, which follows immediately from~\eqref{eq:opt}. The same is true for the square-root bounds. This either follows from the proof of Theorem~\ref{thm:main} and~Remark~\ref{rem:nopositivity}, or from the following analysis based on block matrix algebra. Letting $\bs Y^+=(Y_1,\ldots,Y_d)^\top$, we need to show that 
\[\det\{\Sigma_\mu(X,\bs Y^+)\}/\det\{\Sigma_\mu(\bs Y^+)\}\leq \det\{\Sigma_\mu(X,\bs Y)\}/\det\{\Sigma_\mu(\bs Y)\}.\]
%\[\cov_\mu(X,\overline {\bs Y})\Sigma^{-1}_{\overline {\bs Y}}\cov_\mu(X,\overline {\bs Y})^\top\geq \cov_\mu(X,\bs Y)\Sigma^{-1}_{\bs Y}\cov_\mu(X,\bs Y)^\top,\]
This inequality follows by rewriting it using the Schur complements as in~\eqref{eq:schur} and applying the block matrix inversion formula~\cite[Thm.\ 2.7]{schur_in_stat}. By doing so we find that this inequality is strict unless 
%$\det(\overline \Sigma_{1:d,2:d+1})=0$, that is, the enlarged covariance matrix with last row and first column removed is singular, which is equivalent to
\[\cov_\mu(X,Y_{d})=\cov_\mu(X,\bs Y)\Sigma_\mu(\bs Y)^{-1}\cov_\mu(\bs Y,Y_{d}).\]
In other words, this condition corresponds to the case when the extra moment constraint on~$Y_{d}$ does not improve the square-root bounds, assuming that the enlarged covariance matrix $\Sigma_\mu(X,\bs Y^+)$ is invertible.
%\end{remark}

%One may repeat the steps of the proof for arbitrary $\e_L Z_i=z_i$ to obtain
%\[\lambda^2(\delta-1-(\bs z-\e\bs Z)\Sigma_0^{-1}(\bs z-\e\bs Z)^\top)=\frac{\det(\Sigma)}{\det(\Sigma_0)}.\] 
%Note that $(\bs z-\e\bs Z)\Sigma_0^{-1}(\bs z-\e\bs Z)^\top>0$ for $\bs z\neq \e\bs Z$ which shows that there is no $\lambda>0$ for small $\delta>1$. This is expected, since small $\delta$ does not allow to modify the means of $Z_i$ sufficiently in order to obtain $\bs z\neq \e\bs Z$ under the new measure.

\section{Bounds on Pickands' dependence function}\label{sec:pickands_bounds}

In this section we apply the bounds from the previous sections to
assess the model misspecification error in multivariate extremes with the focus on tail probabilities in~\eqref{eq:approx_main}.
For illustration, we consider the bivariate case and note that the extension
to higher dimensions is analogous. Possible computational challenges will be
discussed in Section~\ref{sec:var} presenting another application of the robust approach to multivariate extremes.

Recall that we exclusively address misspecification of the spectral distribution. In the bivariate
case, this distribution is defined on the simplex $\mathbb S^1$, and it is thus effectively one-dimensional.
In the following we assume that the corresponding random variable $Y=Y_1\in[0,1]$ $\mu$-a.s.
 %We denote it by~$Y$. 
Alternatively to~\eqref{max_prob}, we may directly consider the robust bound on the Pickands' dependence function defined in~\eqref{eq:pickands}:
\begin{equation}\label{eq:optA}
\sup_{\p'}\left\{2\e'\left\{(1-z)Y\vee z(1-Y)\right\}:D_\mu(\p',\p)\leq \delta,\e Y=1/2\right\},
\end{equation}
i.e., we take $X=X(z)=2\{(1-z)Y\vee z(1-Y)\}$ for a fixed $z\in(0,1)$. The respective optimal value $A^*(z;\delta)$ provides the asymptotic robust upper bound on the tail probability
\[ \p(Z_1>tz_1\text{ or }Z_2>tz_2)\lesssim t^{-1}\left(\frac{1}{z_1}+\frac{1}{z_2}\right)A^*\left(\frac{z_1}{z_1+z_2};\delta\right),\qquad t\to \infty\] The lower bound is obtained similarly using the optimal value $A_*(z;\delta)$ of the corresponding minimization problem.

%In contrast to the univariate extreme value distributions with only three parameters,
%the spectral measure $H$ cannot be fully described by a parametric family
%of distributions. Any choice of parametric family as dependence model will thus
%suffer from model misspecification. In the following, we analyse how severe the
%effect of this misspecification can get. More precisely, we are interested
%in the maximal and minimal possible values of $\theta(z)$ over all
%spectral measures in a $\delta$-neighbourhood
%of the reference distribution $H$, measured in terms of the
%divergence $D_\mu$ defined in~\eqref{eq:divergence}. We consider the optimization problem~\eqref{max_prob} which is a particular case of~\eqref{gen_prob} with 
%\begin{equation}\label{eq:X}
%X=X(z)=\frac{Y}{z}\vee \frac{1-Y}{1-z},\qquad Y\in[0,1],\qquad \e Y=1/2
%\end{equation}
%for every $z\in(0,1)$, as well as the corresponding minimization problem.

%In this section we focus on the application of the developed theory to bivariate extremes. Recall that the function $\theta(z),z\in (0,1)$, defined in~\eqref{eq:theta_z} represents the dependence structure of a bivariate max-stable distribution, and it provides asymptotic expressions for joint exceedance probabilities. Thus we consider the optimization problem~\eqref{max_prob} which is a particular case of~\eqref{gen_prob} with $X$ given in~\eqref{eq:X} and $Y\in[0,1]$ satisfying the moment constraint $\e Y=1/2$. 
%We have the following bounds on $\theta(z)$ for a fixed $z$, given divergence level~$\delta>0$ and a chosen dominating measure~$\mu$:

In addition to the exact robust bounds we will use the corresponding square-root bounds, and also bounds in the model class for comparison. More precisely, we consider the following:
\begin{itemize}
\item[(a)] Robust bounds $A^*(z;\delta)$ and $A_*(z;\delta)$ given by Theorem~\ref{thm:opt}, which can be computed as explained in Section~\ref{sec:computation}. Details on the identification of $\delta^{**}$ and $\delta_{**}$ are postponed to the Appendix~\ref{sec:delta_starstar}, because this is not essential for applications.
\item[(b)] Robust square-root bounds $\widehat A^*(z;\delta)\geq A^*(z;\delta)$ and $\widehat A_*(z;\delta)\leq A_*(z;\delta)$ given by~\eqref{eq:sqrt_1}. These are conservative bounds that are easy to compute. Moreover, they are exact when $\delta\leq\delta^*$ and $\delta\leq \delta_*$ for upper and lower bounds, respectively.
\item[(c)] Exact bounds in the model class that are not robust under model misspecification. That is, we impose the restriction that $Y$ under $\p'$ belongs to the chosen model class. These bounds are easy to compute for, e.g., one-parameter families, but otherwise it can be a hard problem. This paper addresses model misspecification issues and so the bounds in the model class will be given only for comparison.
\end{itemize}

%\begin{remark}
%Observe that the exact robust upper bound $A^*(z;\delta)$ is convex in $z\in [0,1]$ as a supremum of convex functions, and it yields 1 for $z\in\{0,1\}$. Hence it is a valid Pickands' functions, see~\eqref{eq:pickands} and comments following it, and so it can be thought of as representing the worst-case spectral measure. It is not immediate that the lower bound and the square-root bounds are convex. Moreover, the square-root bounds are not in $[z\vee(1-z),1]$ for large~$\delta$. 
%%Normally, it should be possible to transform the square-root bounds to be a valid Pickands' function.
%\end{remark}

\begin{remark}
  The bounds in (a) and (b) on Pickands' function directly provide robust bounds on the \emph{extremal coefficient} $\theta = 2A(1/2) \in[1,2]$, a commonly used summary statistic for dependence in multivariate and spatial extreme value statistics \citep{sch2003}.
\end{remark}

Regarding the optimization problem~\eqref{eq:optA} it is convenient to switch to the induced distributions of~$Y$, see also Section~\ref{sec:measure_space}. Thus we assume that $\Omega=[0,1]$ and $\mathcal F$ is the respective Borel $\sigma$-algebra, and that $Y(\omega)=\omega$. In Section~\ref{sec:divergence} we claimed that the choice of the dominating measure $\mu$ reflects our uncertainty about the measure~$\p$. Throughout this paper, however, our main choice is $\mu=\p$ leading to the R\'enyi divergence of order~2. For the purpose of illustration we also use $\mu=$ Leb$[0,1]$ assigning uniform weights; see~\eqref{eq:unif}. In the following we discuss some further simplifications of the general theory in these two particular cases.
% Furthermore, we report the values of $\delta^*$ and $\delta_*$ determining the regions where the square-root bounds (b) coincide with precise bounds~(a). 
 %A detailed study of a choice of the dominating measure $\mu$ is left for future work.

%\begin{align*}
%\theta(\lambda)&=2\int_0^1 \omega\vee(1-\omega)f_\lambda(\omega)\D \omega+2\sum_i\omega_i\vee(1-\omega_i)p_\lambda(\omega_i),\\
%\delta(\lambda)&=\int_0^1\left(\frac{f_{\lambda}(\omega)}{f_{\lambda_0}(\omega)}\right)^2f_{\lambda_0}(\omega)\D \omega+\sum_i\left(\frac{p_{\lambda}(\omega_i)}{p_{\lambda_0}(\omega_i)}\right)^2p_{\lambda_0}(\omega_i)-1.
%\end{align*}
%Optimizing over $\{\theta(\lambda):\delta(\lambda)\leq \delta\}$ we obtain the bounds in the model class. In various examples one may expect that the maximizer $\lambda^*$ and minimizer $\lambda_*$ solve the equation $\delta(\lambda)=\delta$. 

If $\mu=\p$, then computing the robust bounds requires solving a system of two non-linear equations, see Section~\ref{sec:computation} for details.
On the contrary, the square-root bounds \eqref{eq:sqrt_1} are always explicit, and we only need to compute $\e X,\var(X)$ and $\corr(X,Y)$.
Moreover, Remark~\ref{rem:mu=p} provides simple expressions for $\delta^*$ and $\delta_*$ in terms of $b^*$ and~$b_*$. Importantly, the latter two can be given explicitly under a minor assumption that $0,1,z$ are in the support of the distribution of $Y$: 
\begin{align*}
b^*&=\e X-\rho/2+(-2z)\vee (\rho z-2z(1-z))\vee (\rho-2(1-z)),\\
b_*&=\e X-\rho/2+(-2z)\wedge(\rho-2(1-z)),
\end{align*}
where $\rho=\cov(X,Y)/\var(Y)$.
Indeed,
\[b^*=\e X-\rho/2+{\rm ess\, sup}(\rho Y-X),\qquad b_*=\e X-\rho/2+{\rm ess\, inf (\rho Y-X)}\]
according to Remark~\ref{rem:mu=p}, and $\rho Y-X$ achieves its maximum in one of the points $Y=0,Y=z$ or $Y=1$, and its minimum in $Y=0$ or $Y=1$.
Further comments concerning $\delta^{**}$ and $\delta_{**}$ will be given in Section~\ref{sec:delta_starstar}.

If $\mu=\Leb$, then for the exact robust bounds we need to solve a system of three non-linear equations. Concerning the square-root bounds, we observe from~\eqref{eq:sqrt_1} that their width is determined by the dominating measure $\mu$, whereas $L$ affects the center $\e X$ and the values of $\delta^*,\delta_*$ only; see~\eqref{eq:condition} and~\eqref{eq:condition_min}.
A simple calculation based on~\eqref{eq:sqrt_1} yields the following square-root bounds for an arbitrary density~$L$:
\begin{equation}\label{eq:sqrt_leb}\e X\pm \sqrt{\delta\frac{4}{3}z^3(1-z)^3}.\end{equation}
%Note that $\delta$ represents the divergence of the true model from the reference model and so the square-root bounds change accordingly.
Furthermore, for any $\delta$ one can provide a lower bound on the density $L(w)$ so that the square-root upper and lower bounds are exact. In particular, one can show that if $L$ approaches 0 at one of the ends of the interval $[0,1]$ then $\delta_*=0$, that is, the lower square-root bound is never exact.
Similarly, if $L$ approaches~0 at~$z$ then~$\delta^*=0$.

%In particular, for the exact square-root upper bound $L$ must be away from~0 on some interval straddling~$z$, whereas for the exact square-root lower bound $L$ must be away from 0 at both ends of $[0,1]$.

%In order to determine the exact upper bound we need to solve the following system of two equations:
%\begin{align*}
%&\int_0^1 \left(\lambda y\vee(1-y)+\lambda\mu+f(y)\right)_+\D y=1,\\
%&\int_0^1 \left((\lambda y\vee(1-y)+\lambda\mu+f(y))_+-f(y)\right)^2\D y=\delta.
%\end{align*}
%\subsection{Zero covariance case}
%\todo{We need it?}
%Consider for a moment a special case when $z=1/2$ and assume that he $\mu$-distribution of $Y$ is symmetric around $1/2$. In this case we obtain $\cov_\mu(X,Y)=0$. 
%It turns out that this case leads to various important simplifications.
%Firstly, $\cov(f(X),Y)=0$ for any function $f$ such that $f(X)$ is square integrable, which follows from the identity
%\begin{equation*}\e(f(Y\vee(1-Y))Y)=\e(f(Y\vee(1-Y))(1-Y))\end{equation*}
%and the fact that $\e Y=1/2$. In particular, we have $\cov(X,Y)=0$ and so~\eqref{eq:bound_conservative} reads
%\begin{equation}\label{eq:bound_conservative_sym}
%V_2(\delta)=\theta+\sqrt{\delta\var(X)},\qquad\delta\leq \delta^*.
%\end{equation}
%Furthermore, we have
%\[\delta^*=\var(X)/(\theta-1)^2,\qquad \delta_*=\var(X)/(\theta-2)^2.\]
%Note that our assumption on the support of $Y$ does not allow for $\theta=1,2$, because in these cases the distribution of $Y$ must put all its mass on $Y=1/2$ or on $Y=0\text{ or }Y=1$.

\subsection{Illustration of the bounds}\label{sec:illustration}
In the beginning of this section we provided a list of three bounds on the Pickands' dependence function~$A(z)$: (a) exact robust bounds, (b) conservative square-root bounds and (c) bounds in the model class.
Let us illustrate these bounds for different divergence levels~$\delta$ with an example of a H\"usler--Reiss spectral distribution; see Appendix~\ref{appendix_distributions} for several common parametric families of spectral distributions. 
The H\"usler--Reiss distribution has a single parameter $\lambda\in(0,\infty)$ and we fix it to~$\lambda=0.6$. Furthermore, we consider $z=0.4$ and use two dominating measures: $\mu=\p$ and $\mu=$~Leb. The upper panels of Figure~\ref{fig:HR_densities} show the bounds as functions of divergence~$\delta$. The middle and lower panels depict the H\"usler--Reiss density $h_{0.6}(\omega)$ as well as the optimizing densities corresponding to the upper and lower bounds for a particular choice of~$\delta=0.4$. In order to make comparisons easier, the densities with respect to the Lebesgue measure are depicted in both cases, and so for $\mu=\p$ we plot $L^*(\omega)h_{0.6}(\omega)$ and $L_*(\omega)h_{0.6}(\omega)$ rather than $L^*(\omega)$ and $L_*(\omega)$. 
Finally, there is no density corresponding to the square-root bound when $\delta$ is larger than $\delta^*$ or $\delta_*$ for the upper and lower bound, respectively. Nevertheless, there always exists the corresponding pseudo-density which is not necessarily non-negative; see Remark~\ref{rem:nopositivity}. These pseudo-densities are also included in Figure~\ref{fig:HR_densities}.

\begin{figure}[h!]
\begin{center}
\begin{tabular}{cc}
$\mu=\p$&$\mu=$~Leb\\
\includegraphics[width=0.47\textwidth,trim= 10mm 10mm 10mm 0mm]{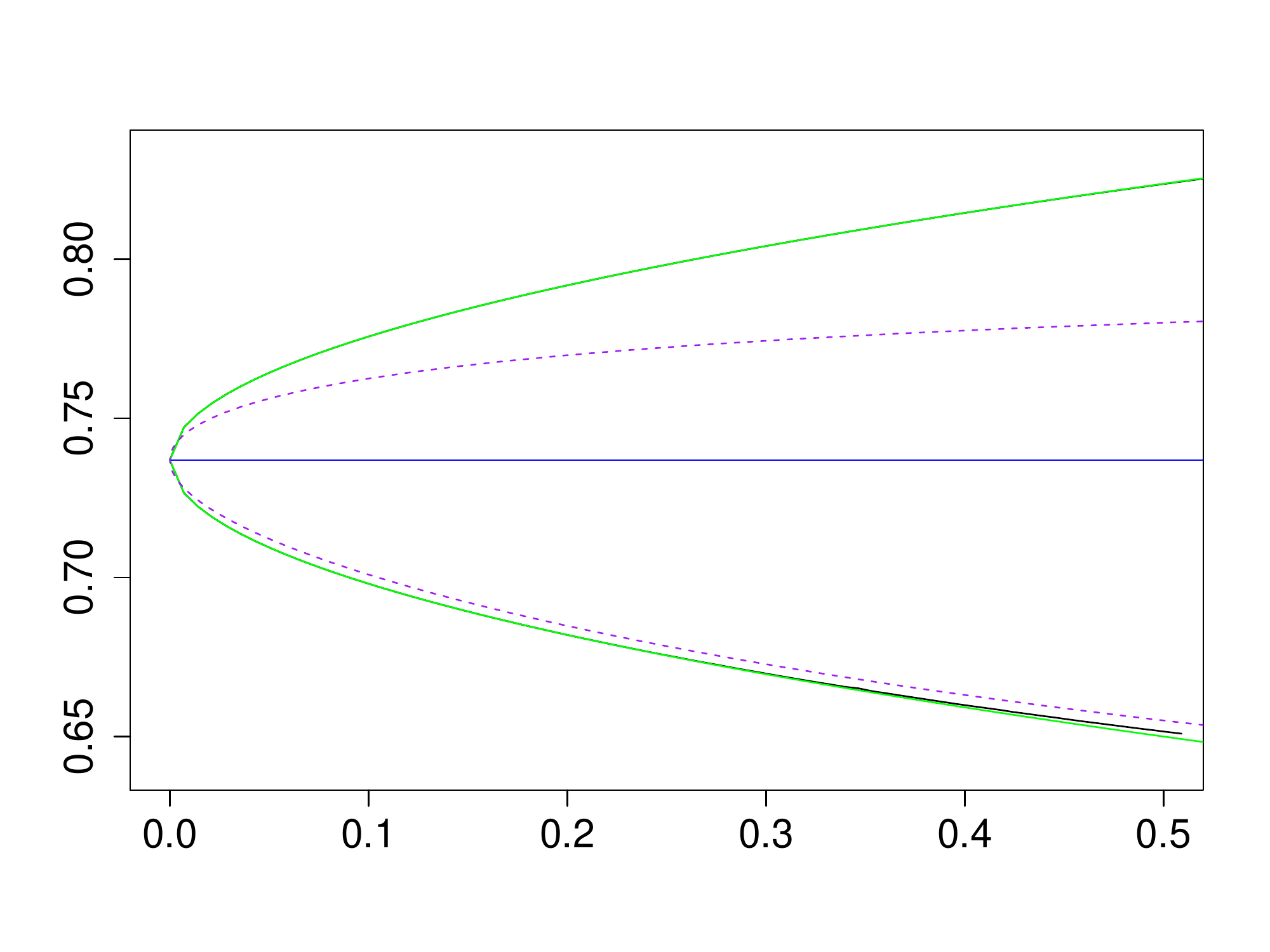}
&\includegraphics[width=0.47\textwidth,trim= 10mm 10mm 10mm 0mm]{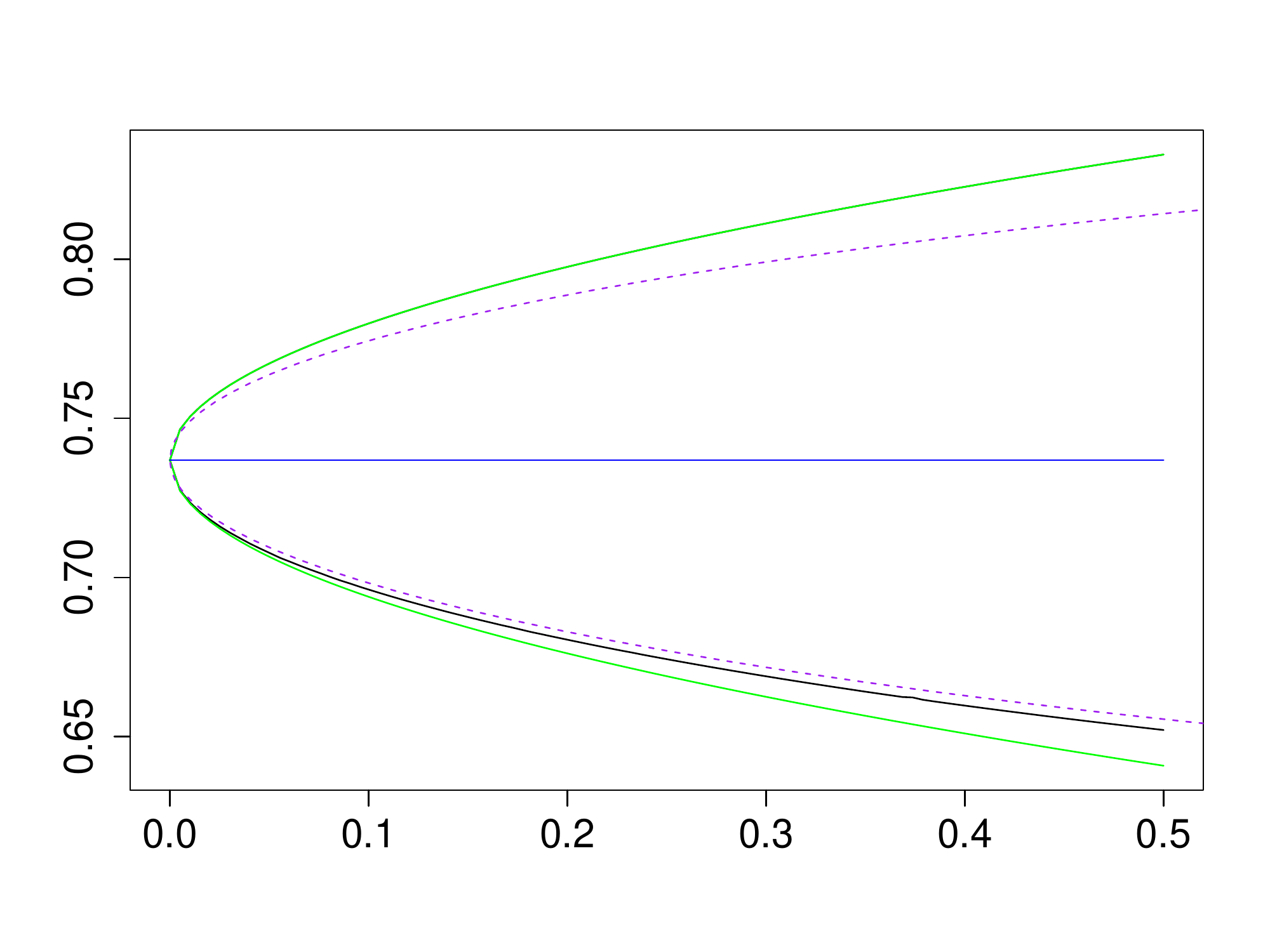}\\
%\multicolumn{2}{c}{Bounds as a function of $\delta$.}\\
\includegraphics[width=0.47\textwidth,trim= 10mm 10mm 10mm 10mm]{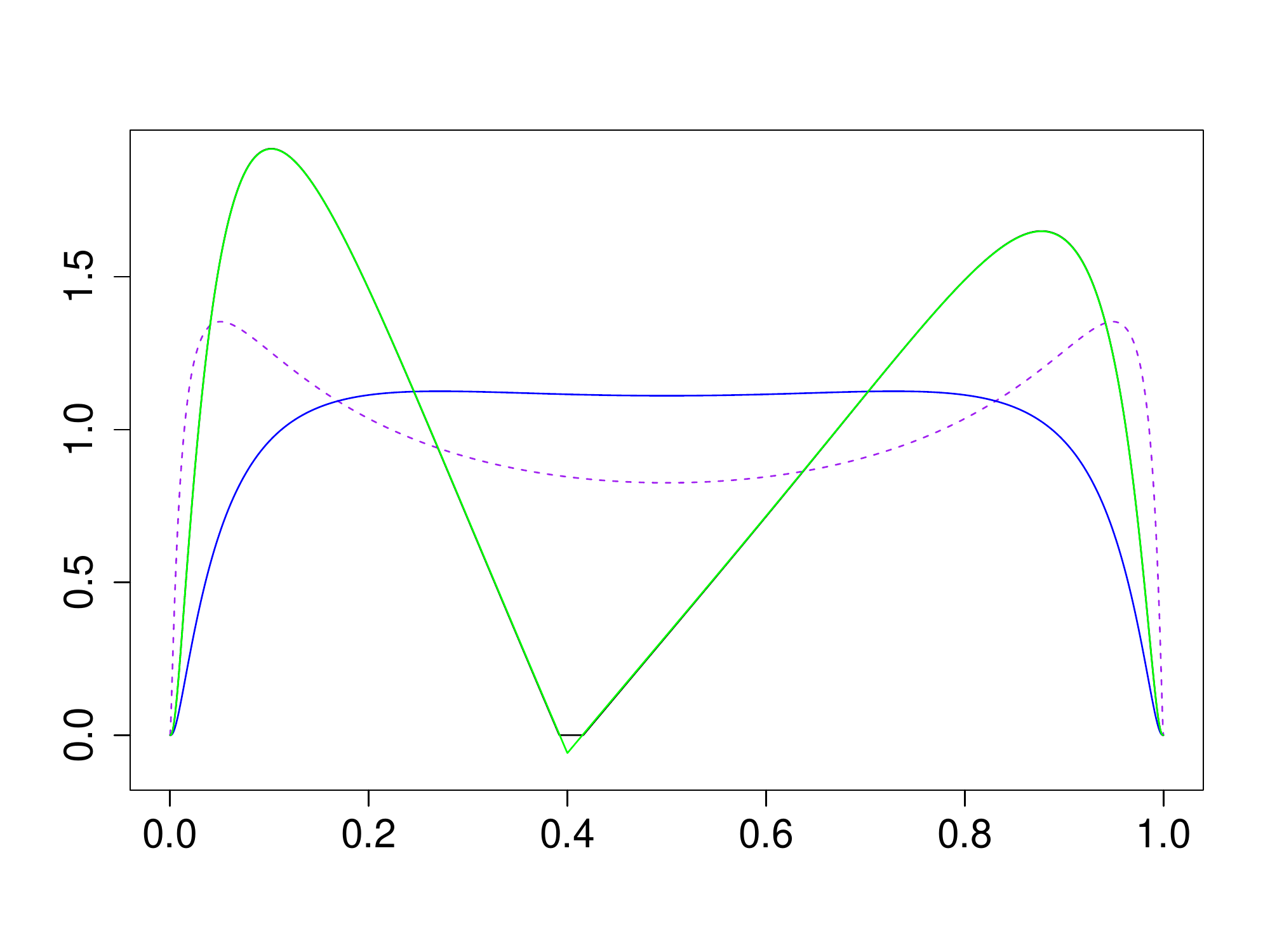}
&\includegraphics[width=0.47\textwidth,trim= 10mm 10mm 10mm 10mm]{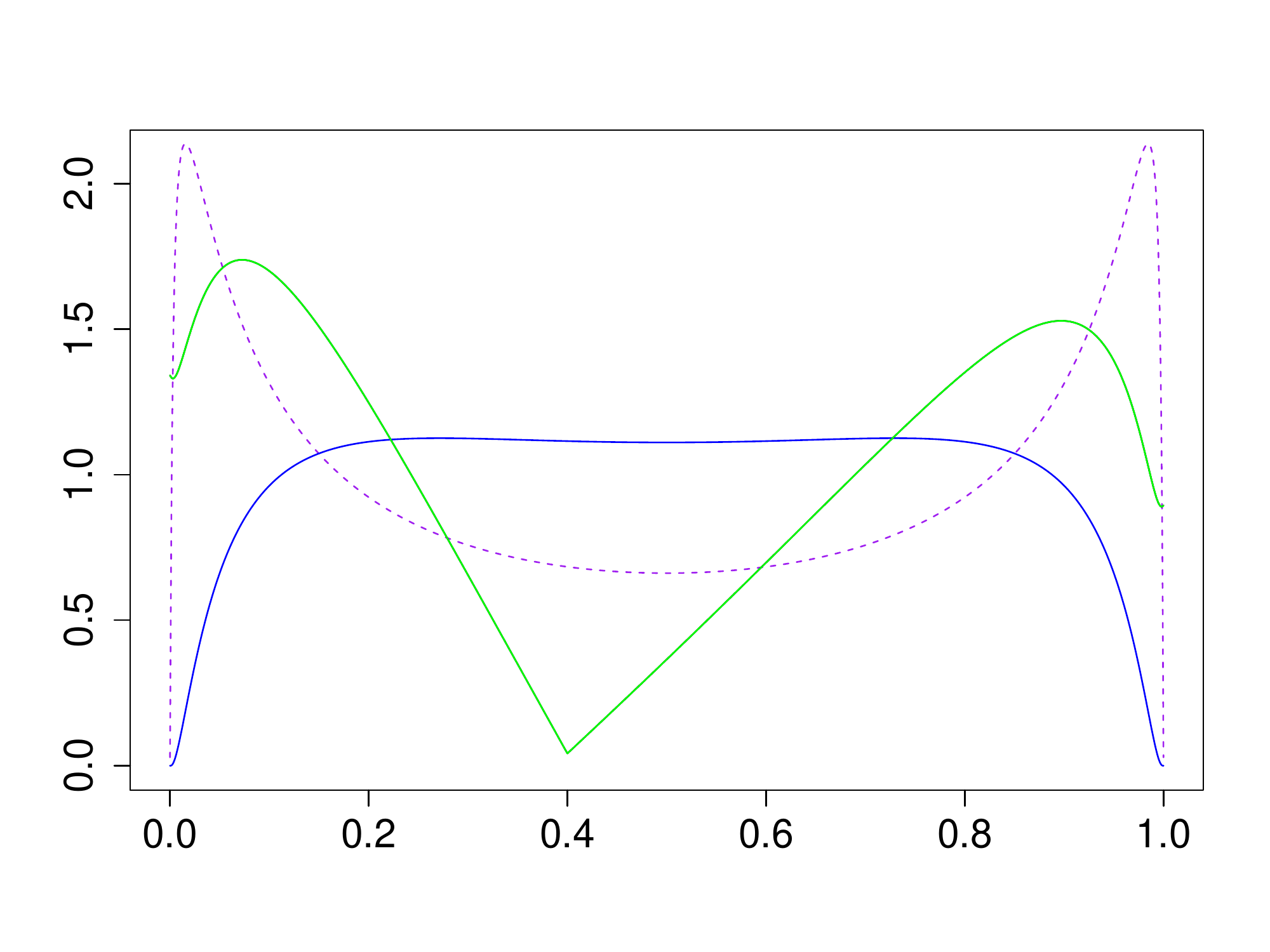}\\
%\multicolumn{2}{c}{Densities for the upper bounds when $\delta=0.4$.}\\
\includegraphics[width=0.47\textwidth,trim= 10mm 10mm 10mm 10mm]{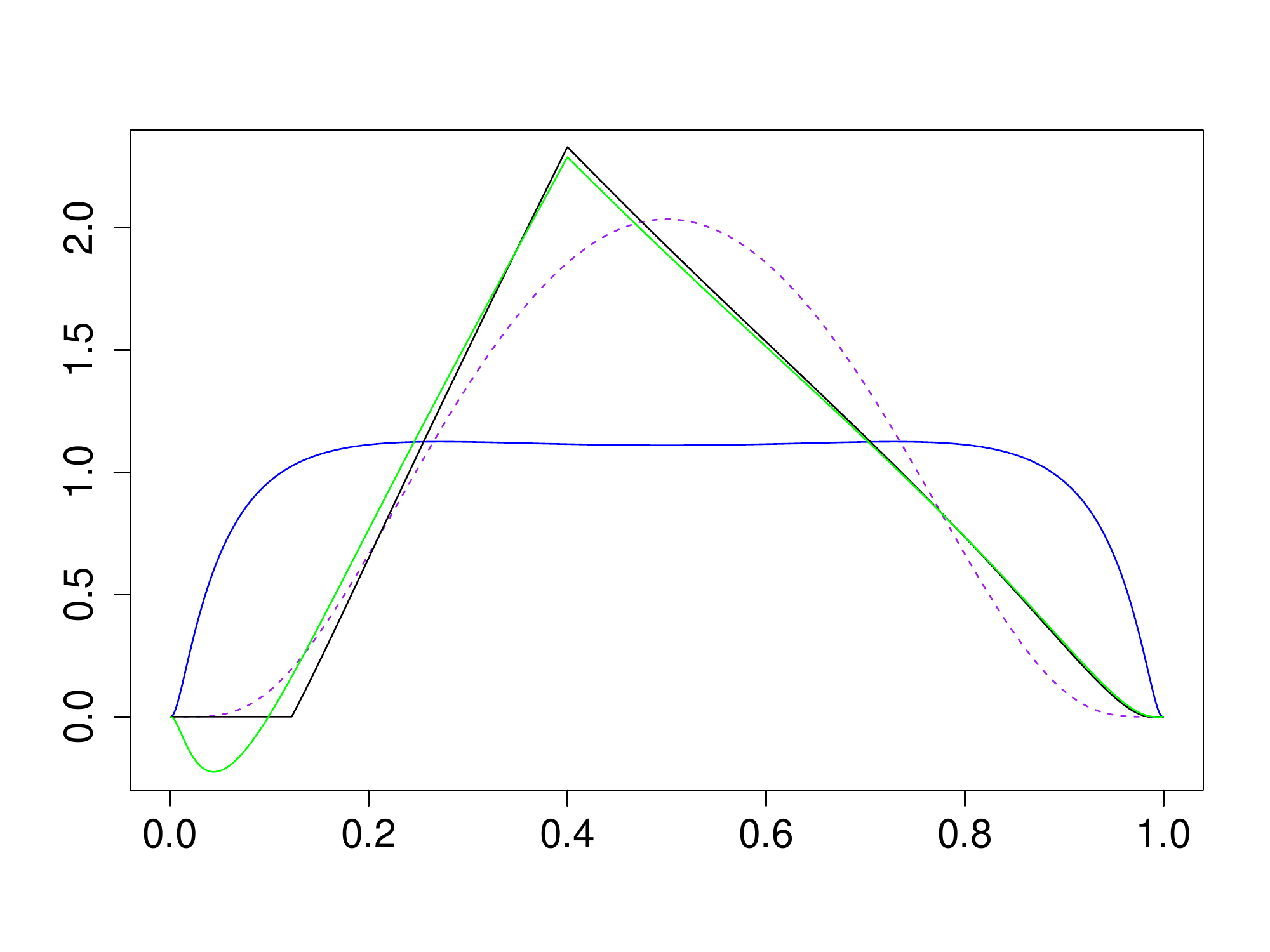}
&\includegraphics[width=0.47\textwidth,trim= 10mm 10mm 10mm 10mm]{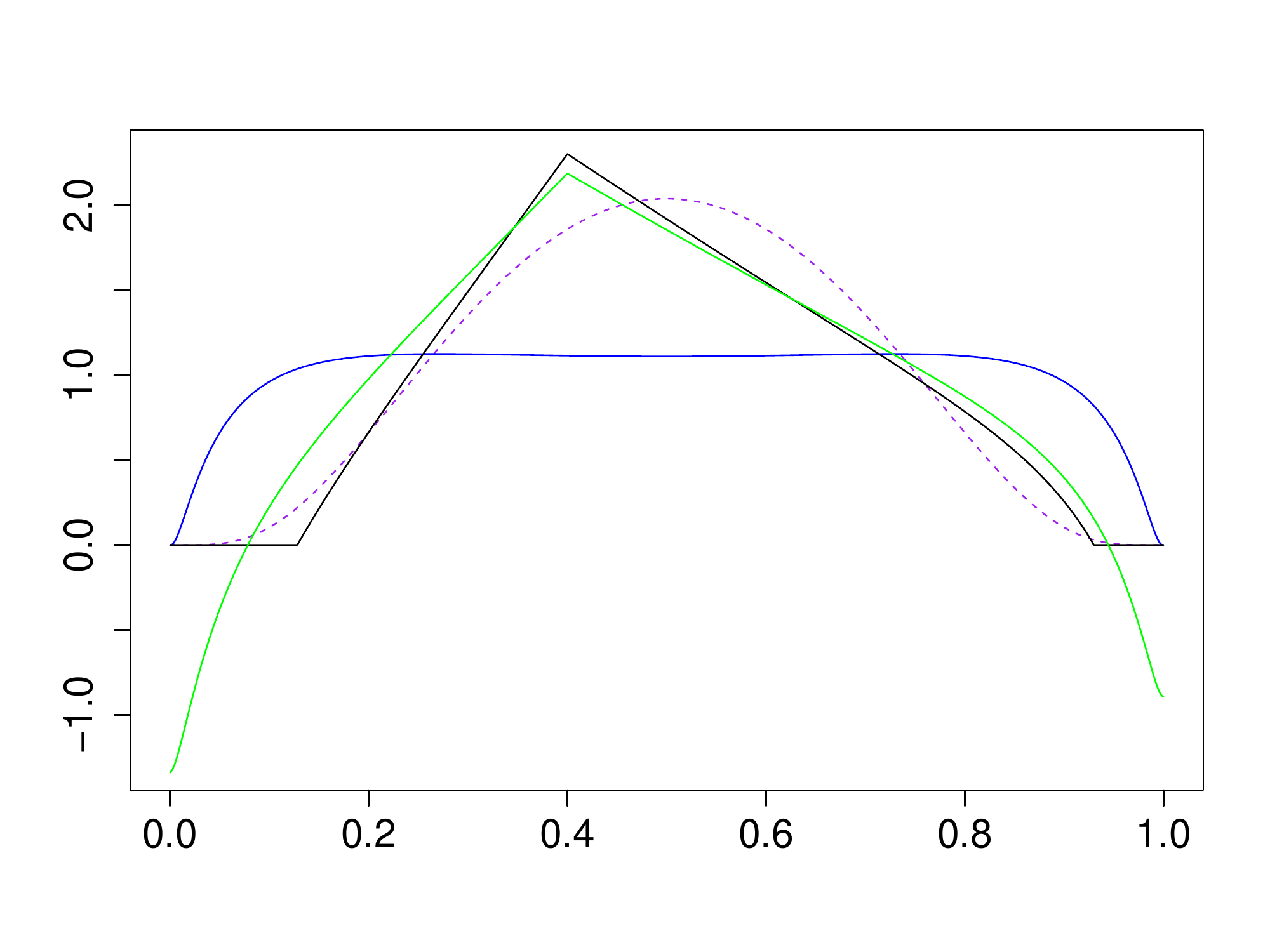}\\
%\multicolumn{2}{c}{Densities for the lower bounds when $\delta=0.4$}.
\end{tabular}
\caption{Upper panels contain the value of $A(z)$ for $\p$ being H\"usler--Reiss with $\lambda=0.6$ and $z=0.4$ (blue), exact robust bounds (black), the square-root bounds (green) and bounds in the model class (dashed purple) as functions of $\delta$. Middle and lower panels show the H\"usler--Reiss density and the optimizing (Lebesgue-) densities for $\delta=0.4$ corresponding to upper and lower bounds, respectively. The pseudo-densities corresponding to square-root bounds are given in green.} %Purple circles correspond to $\lambda=0.5$ and $\lambda=0.7$.}
\label{fig:HR_densities}
\end{center}
\end{figure}
Additionally, we find $\delta^*=0.36,\delta_*=0.14$ when $\mu=\p$, and $\delta^*=0.43,\delta_*=0$ when $\mu=$~Leb, respectively. 
Thus in the case of $\mu=$ Leb the upper square-root bound is exact for the chosen level~$\delta=0.4$ and so $L^*$ and the corresponding pseudo-density coincide.
In the other cases, exact and square-root bounds do not coincide, but it can be seen that the square-root bound is still a very good approximation of the exact robust bound even when $\delta$ is much larger than $\delta^*$ or $\delta_*$. 
% In order to facilitate such a comparison in Figure~\ref{fig:HR_densities} we plotted purple circles corresponding to $\lambda=0.6\pm 0.1$.
Furthermore, the upper bounds in the model class are obtained for $\lambda=0.737$ and $\lambda=0.844$ according to $\mu=\p$ and $\mu=$~Leb, and the lower bounds for $\lambda=0.367$ and $\lambda=0.366$. 
%Finally, 
%we note that one should be careful when comparing divergences for different dominating measures~$\mu$.

Let us make some final observations concerning the optimizing densities. Firstly, when maximizing $A(z)$ the probability mass is shifted from around $z$ towards 0 and~1.
Conversely, when minimizing $A(z)$ the mass is shifted towards~$z$. Secondly, when $\mu=\p$ the density corresponding to the exact upper bound approaches 0 at both ends, and it does not do so when $\mu=$~Leb. The reason is that the chosen H\"usler--Reiss spectral density decays faster than any power at~0 and at~1, and so R\'enyi divergence as defined in~\eqref{eq:divergence} is finite only if the density of $\p'$ decays fast at 0 and at~1. This issue will arise again in Section~\ref{sec:exp} describing our experiments.

\subsection{Experiments}\label{sec:exp}

In this section we show how the robust bounds are able to capture correctly 
the uncertainty due to model misspecification in a statistical estimation problem.
They remain reliable in situations where classical confidence bounds would underestimate
the statistical error.

As an application of our results we consider the estimation of
tail probabilities of the bivariate, regularly varying random vector $\bs Z$; see Section \ref{sec:reg_var}.
Throughout this section we assume that $\bs Z$ follows an asymmetric logistic 
distribution with dependence parameter $a\in(0,1)$ and asymmetry parameters 
$b_1, b_2\in [0,1]$ as defined in Appendix~\ref{sec:AL}, so that the limiting spectral measure~$H=$ AL$(a,b_1,b_2)=\p_{\textup{true}}$ in~\eqref{eq:limH} has the density~\eqref{eq:AL_dens}.
%The true limiting exponent measure $\nu$ then has spectral distribution $Y$ given in ??. 
We conduct several experiments
that illustrate the use of the robust bounds in practical applications.
All our experiments are carried out according to the following scheme.
\begin{enumerate}[(a)]
\item Simulate $n$ data $\bs Z^{(1)}, \dots, \bs Z^{(n)}$ from a bivariate asymmetric logistic distribution using the $R$-package~\cite{ste2002}. 
\item Transform the samples to polar coordinates as in Section~\ref{sec:reg_var}, and choose $r>0$ such that
there are $k<n$ of the radii exceeding the threshold~$r$. According to~\eqref{eq:limH}, the corresponding angles,
say $Y^{(1)},\dots, Y^{(k)}$, are approximate realizations of the spectral distribution.
The choice
of the threshold $r$ is a trade-off between the sample size $k$ and the approximation error.
\item Choose a parametric family for the spectral distribution and fit it to the observations $Y^{(1)}, \dots, Y^{(k)}$, using maximum likelihood estimation. The parametric family can either
  be the correct asymmetric logistic model, or a misspecified model such as the H\"usler--Reiss
  or the extremal-$t$ described in Appendices~\ref{sec:HR} and \ref{sec:ET}, respectively. This model of the spectral distribution defines our probability measure~$\p$.
\item Plot the Pickands' dependence functions $A_{\textup{true}}$ and $A$
  corresponding to the true asymmetric logistic model~$\p_{\textup{true}}$
  and the estimate~$\p$ from (c), respectively; see~\eqref{eq:pickands}.
\item Estimate the divergence of the data from the fitted model $\delta=D_\mu(\p_{\textup{data}},\p)$ using a naive approach: estimate the density (and point masses) from the given $k$ observations and plug it into~\eqref{eq:divergence_ref} together with the model density from (c). Alternative methods for divergence estimation can be found in, e.g.,~\cite{poczos_divergence}. For comparison, we also compute the true divergence $\delta_{\textup{true}}=D_\mu(\p_{\textup{true}},\p)$ of the true underlying asymmetric logistic distribution from the fitted model. 
\item Compute the robust square-root bounds $\widehat A^*(z;\delta)$ and $\widehat A_*(z;\delta)$ for Pickands' function $A$ using~$\delta$ computed in (e). The exact robust bounds, which are considerably harder to compute, are very close to the square-root bounds and we omit them for clarity of the plots.
\item Compute the classical $95\%$-confidence bounds for the Pickands' function by non-parametric bootstrap. 
This is based on 300 estimates of the model parameters as in (c), each for a resample of the data with replacement. We plot these bounds around $A$.
\end{enumerate}

\begin{remark}\label{rem:AL_sim}
Let us remark that instead of simulating data from the asymmetric logistic distribution we could have used any bivariate distribution from its max-domain of attraction, because we rely on a limiting result in (b) to approximate realizations of~$Y$.
Importantly, it is the limiting asymmetric logistic distribution and the corresponding spectral distribution of~$Y$ which are of main interest since they provide a way to extrapolate tail probabilities out of the sample.
\end{remark}

\begin{table}[h!]
\begin{tabular}{l l l l l l l}
$\#$ & $\p_{\textup{true}}$ &$n$ &Fitted model~$\p$ &$\mu$ &$\delta$ &$\delta_{\textup{true}}$\\
\hline
1  &AL$(0.4,0.7,1)$ &20000 &ET$(0.65,1.21)$ &$\p$ &0.34 &0.35\\
2  &AL$(0.5,1,1)$ &20000 &HR$(0.61)$ &Leb &0.05 &0.06\\
3  &AL$(0.5,1,1)$ &2000 &ET$(0.88,3.37)$ &$\p$ &0.05 &0.05\\
4  &AL$(0.5,0.9,0.5)$ &20000 &AL$(0.49,0.85,0.54)$ &$\p$ &0.02 &0.02\\
\end{tabular}
\caption{Details on the four experiments. AL: asymmetric logistic, HR: H\"usler--Reiss, ET: extremal-$t$.}\label{table:exp}
\end{table}

%#1: 0.01 0.99 "Lower correct for:" 0.46 0.54
%#2:
%3:0.02 0.98 "Lower correct for:" 0.25 0.98

The basic information on the four experiments is given in Table~\ref{table:exp}, and the corresponding plots are given in Figure~\ref{fig:plots}. In all experiments we use $k=500$ exceedances.
In experiment $\#1$ we fit a symmetric extremal-$t$ model to an asymmetric logistic model, where both allow for point masses at 0 and 1. The bootstrap confidence bounds are quite tight 
in this example, but they do not contain the true model on most of the domain.
This shows that these classical bounds are overly confident if the fitted model
is misspecified. The robust bounds, one the other hand, are wider and they do contain
the true model everywhere. We note that the square-root bounds may go outside the triangle
of admissible Pickands' functions, see first row in Figure~\ref{fig:plots}. This, however, can be easily fixed by simply restricting the bounds to stay inside the triangle.
%In general, this can easily be fixed by restricting the bounds to stay inside the interval of 
%the overall minimal and maximal value of the corresponding optimization problem; see Remark~\ref{rem:case_ii}.

In experiment $\#2$ we fit a symmetric H\"usler--Reiss model to a symmetric logistic model ($b_1 = b_2 = 1$) with no point masses. Here the dominating measure is $\mu=$~Leb for the reasons that we discuss below. Even though both models are symmetric, the H\"usler--Reiss family is not flexible enough to well-approximate the generating logistic distribution. This is underlined by the
fact that the Pickands' function does not stay inside the bootstrap bounds, but only inside the wider robust bounds.

In the first two experiments we simulate $n=20000$ data points, corresponding roughly to 55 years of daily observations. We choose $r$ to be the $97.5\%$ quantile of all radii, and use for fitting the
$k=500$ observations whose radii exceed $r$. In experiment $\#3$ 
we only have $n=2000$ data points and still use $k=500$, corresponding to the $75\%$
quantile for $r$. Comparing the histograms $\#2$ and $\#3$ in Figure~\ref{fig:plots},
we note that in the latter case there are less observations close to 0 and 1. This 
illustrates that the data used for fitting comes from a pre-limit distribution.
The $\delta$ we are estimating therefore represents the divergence of the data, that is,
the pre-limit distribution, from the fitted model. This number can be considerably
larger or smaller than $\delta_{\textup{true}}$, the divergence of the generating logistic
distribution from the fitted model. 
In this case the robust theory still works well, but the estimated~$\delta$ becomes unreliable. In experiment $\#3$ we chose a run with similar divergences $\delta$ and~$\delta_{\textup{true}}$.

Experiment $\#4$ shows the case of fitting the well-specified asymmetric logistic family
to the data. As expected, both the bootstrap and the robust bounds contain the true model.
It is interesting to observe that both types of bounds almost coincide, meaning
that the robust version is not overly conservative in the well-specified case.

\begin{figure}[p]
\begin{center}
\begin{tabular}{cc}
\includegraphics[width=0.47\textwidth,trim= 10mm 10mm 10mm 10mm]{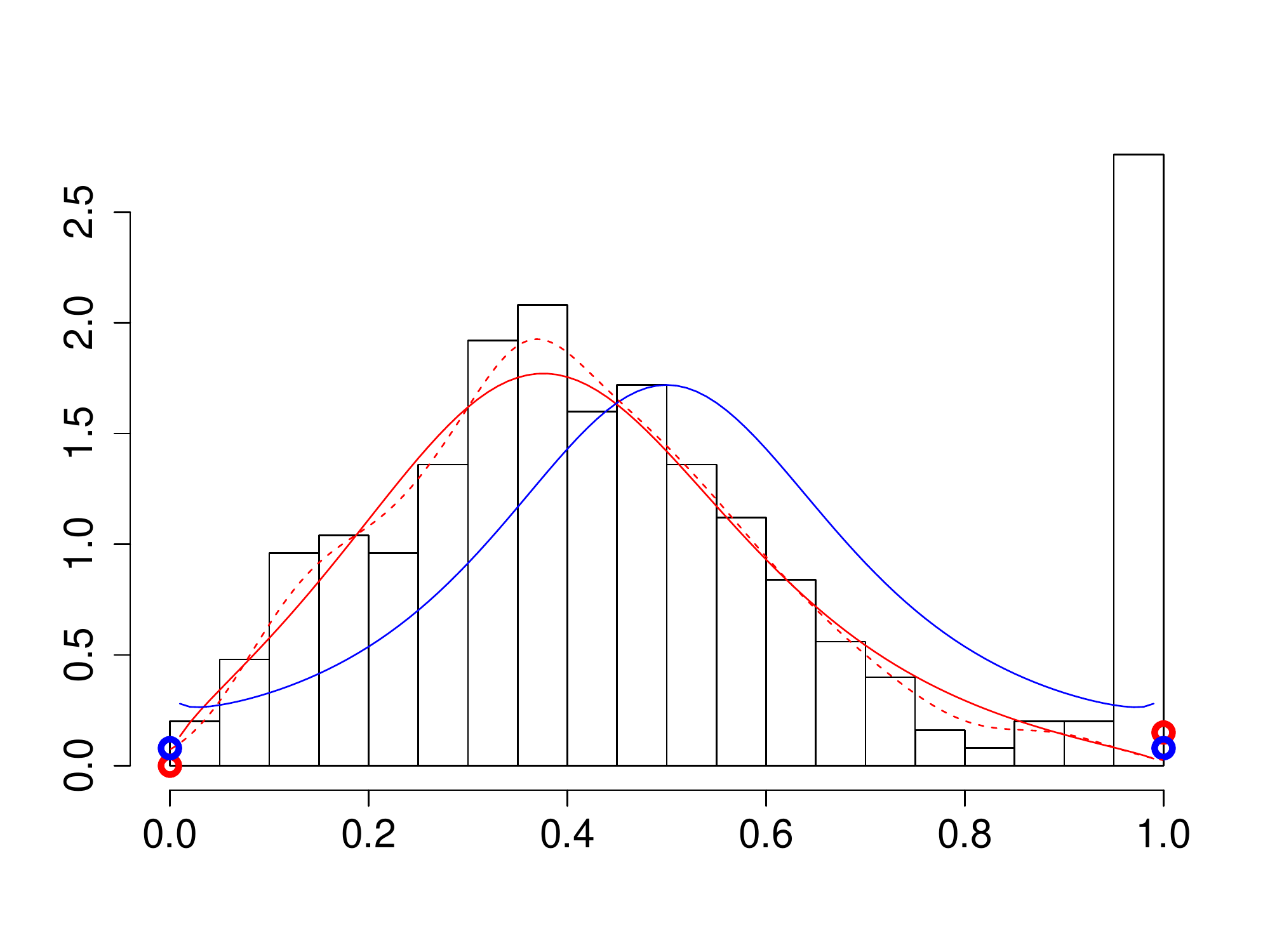}
&\includegraphics[width=0.47\textwidth,trim= 10mm 10mm 10mm 10mm]{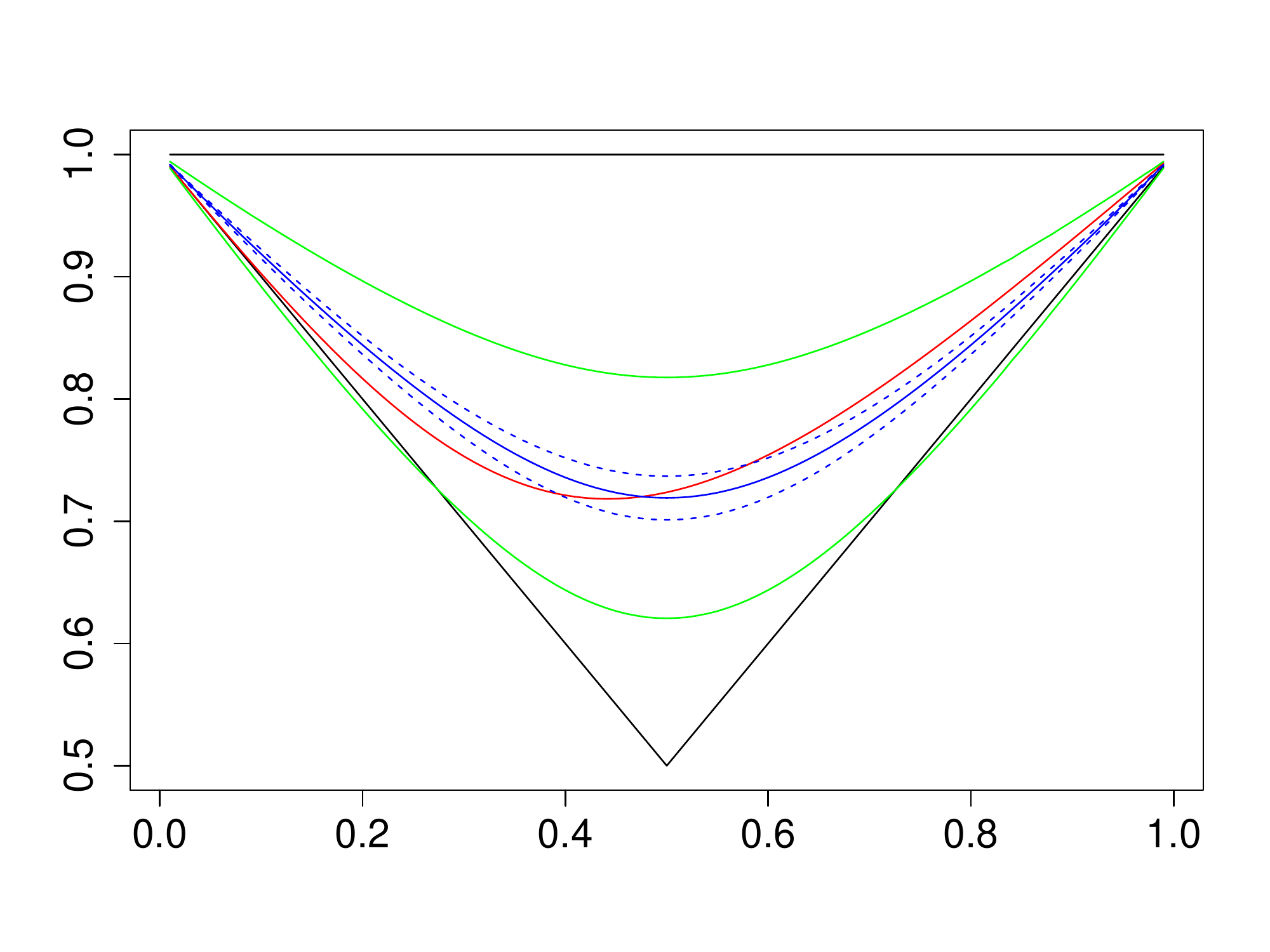}\\
\includegraphics[width=0.47\textwidth,trim= 10mm 10mm 10mm 10mm]{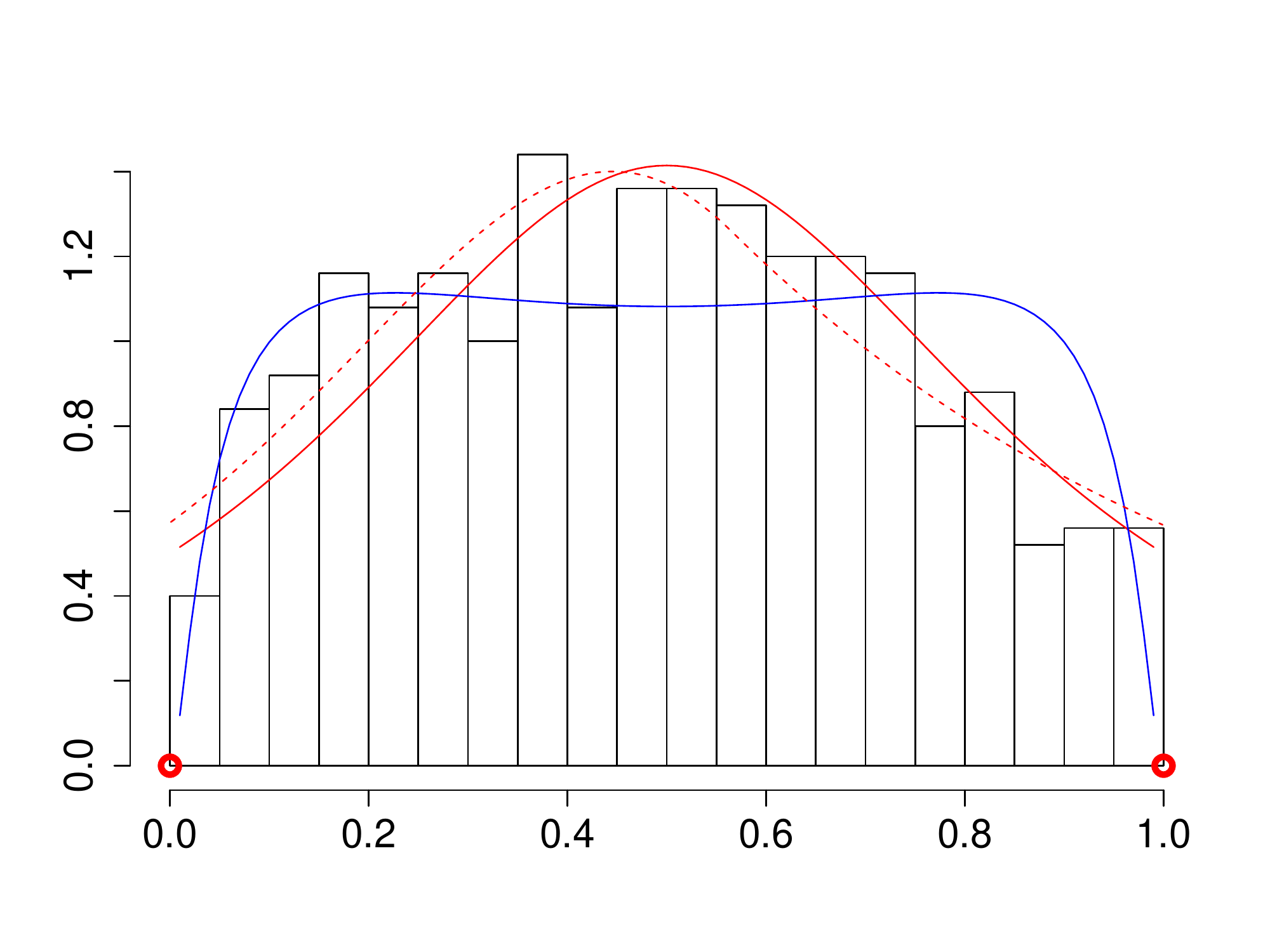}
&\includegraphics[width=0.47\textwidth,trim= 10mm 10mm 10mm 10mm]{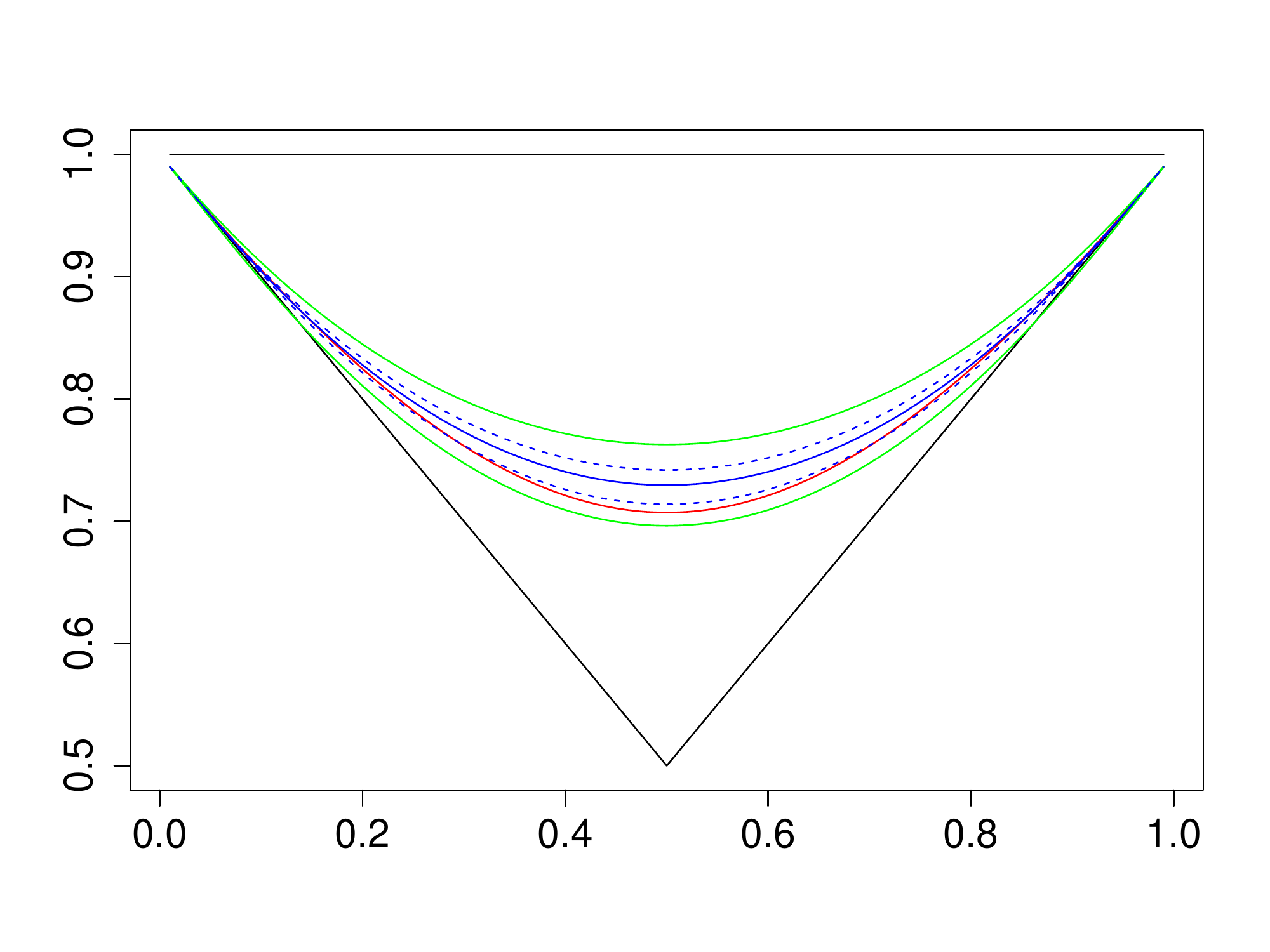}\\
\includegraphics[width=0.47\textwidth,trim= 10mm 10mm 10mm 10mm]{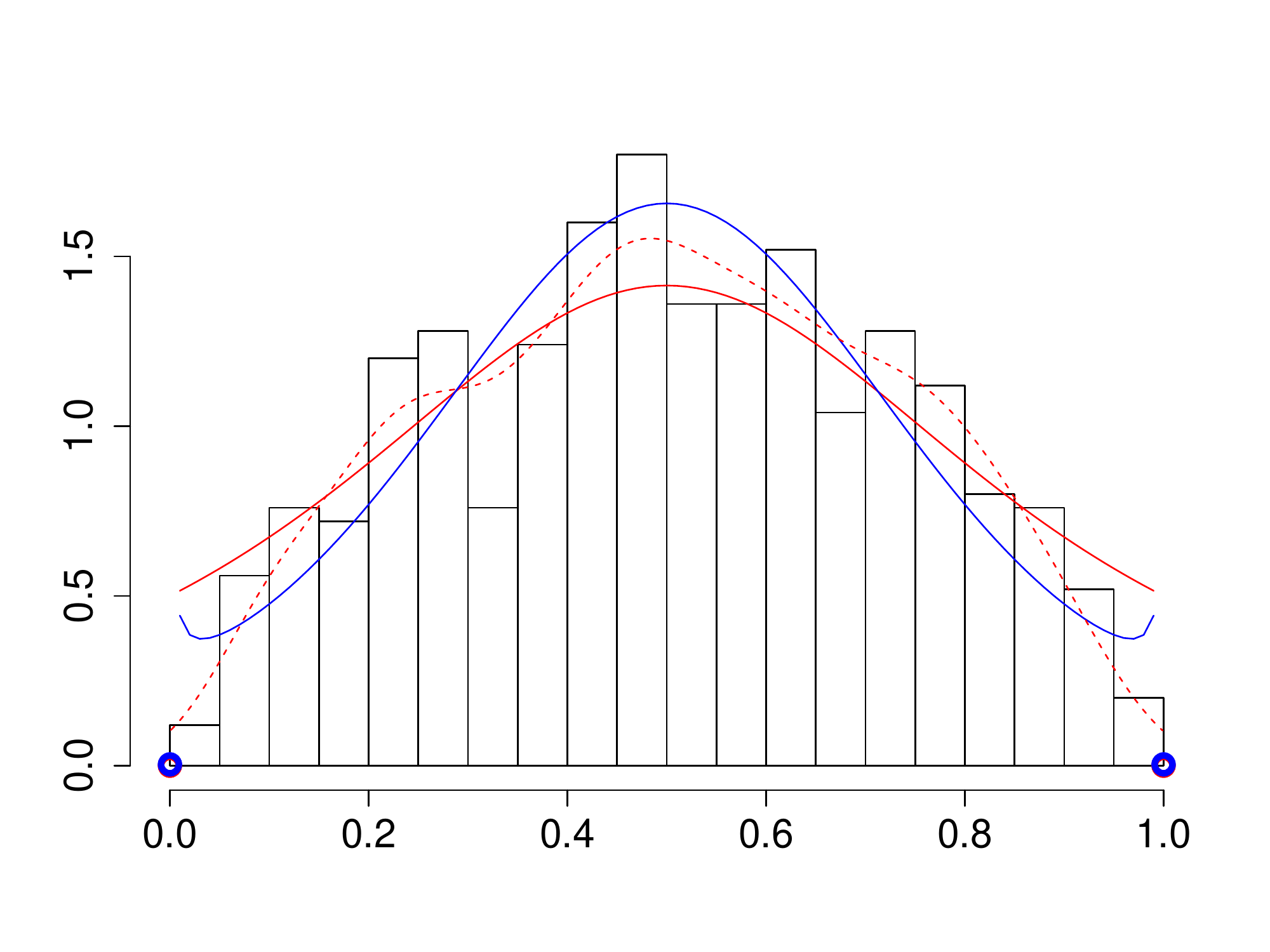}
&\includegraphics[width=0.47\textwidth,trim= 10mm 10mm 10mm 10mm]{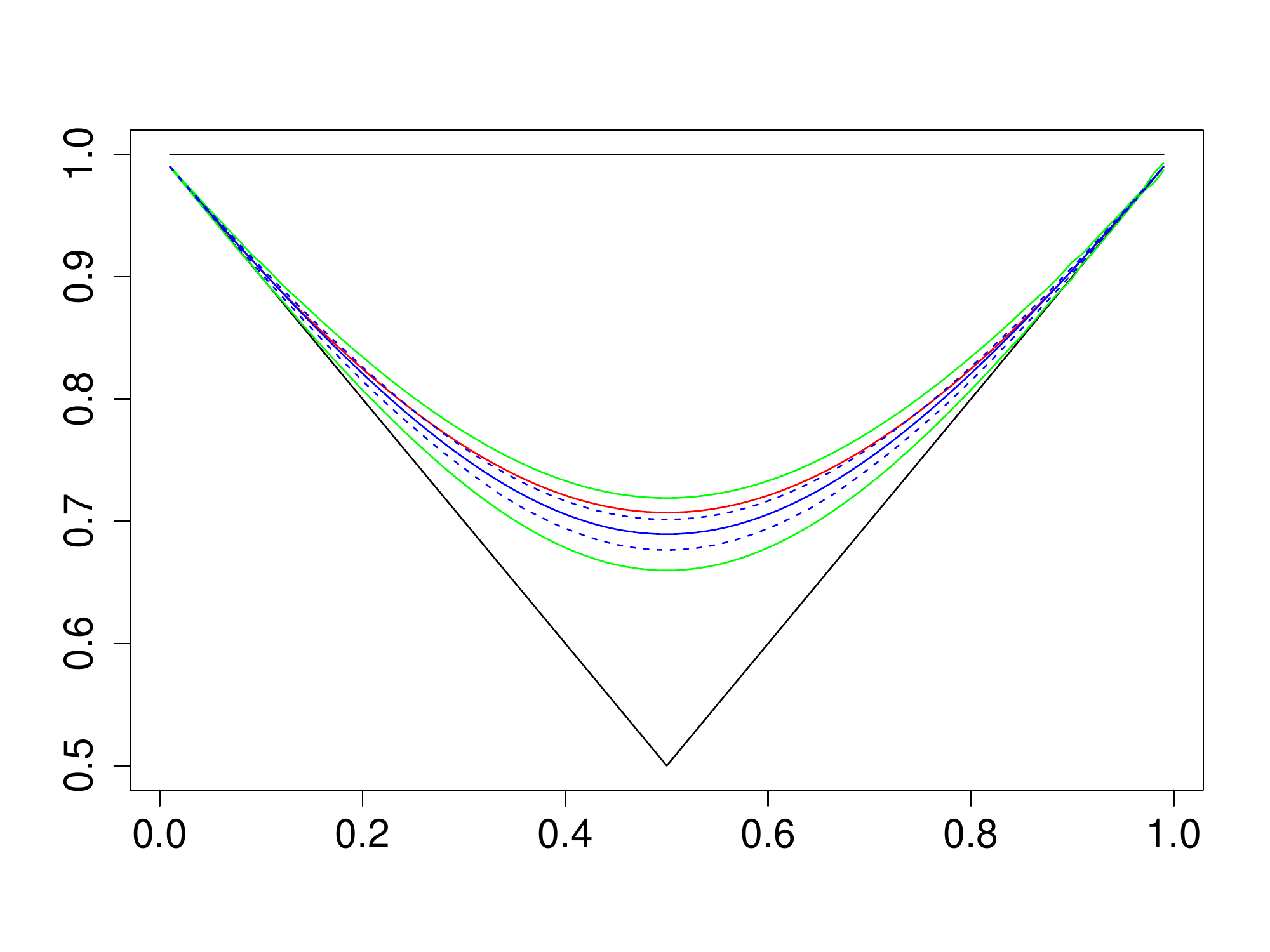}\\
\includegraphics[width=0.47\textwidth,trim= 10mm 10mm 10mm 10mm]{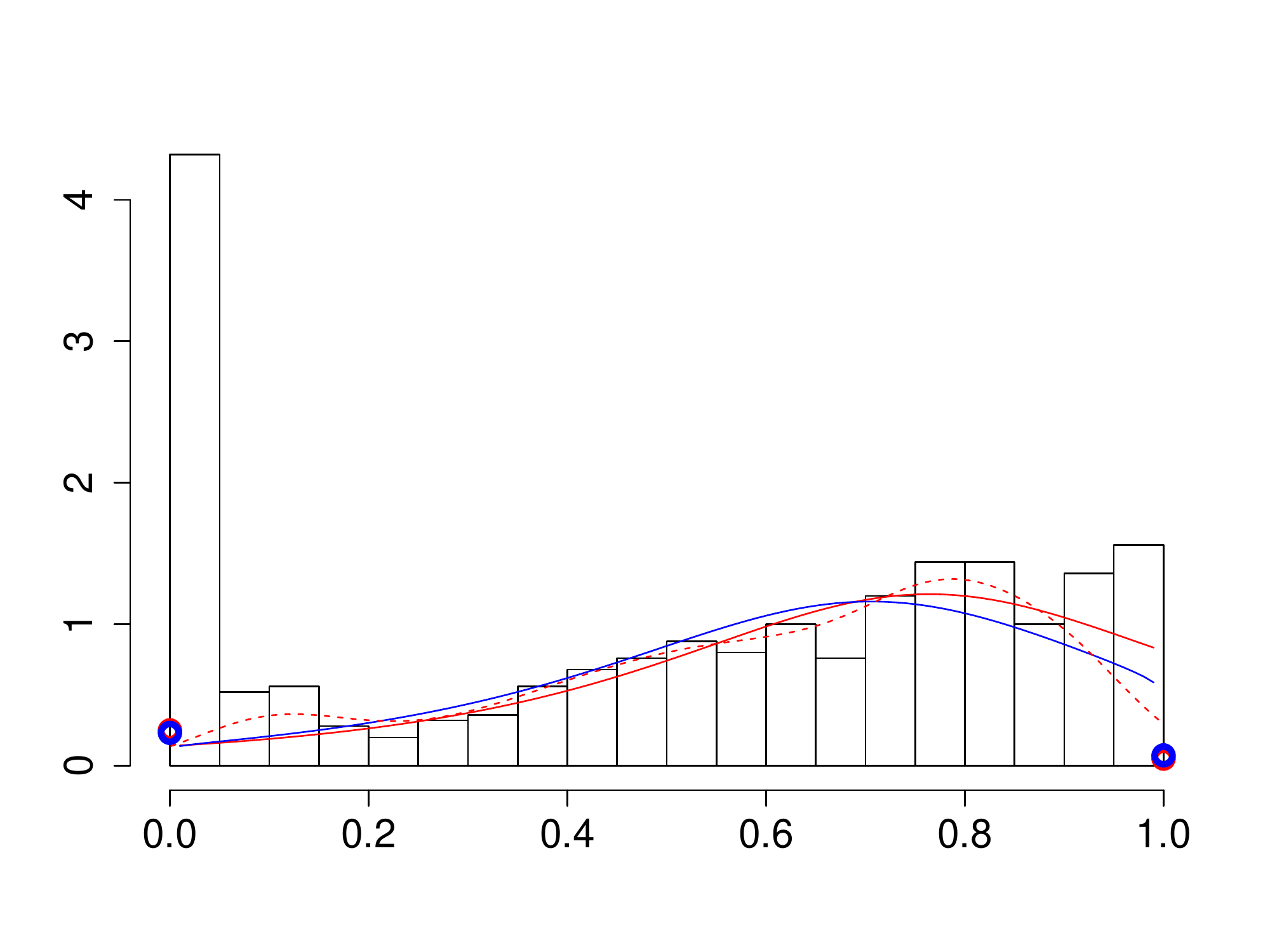}
&\includegraphics[width=0.47\textwidth,trim= 10mm 10mm 10mm 10mm]{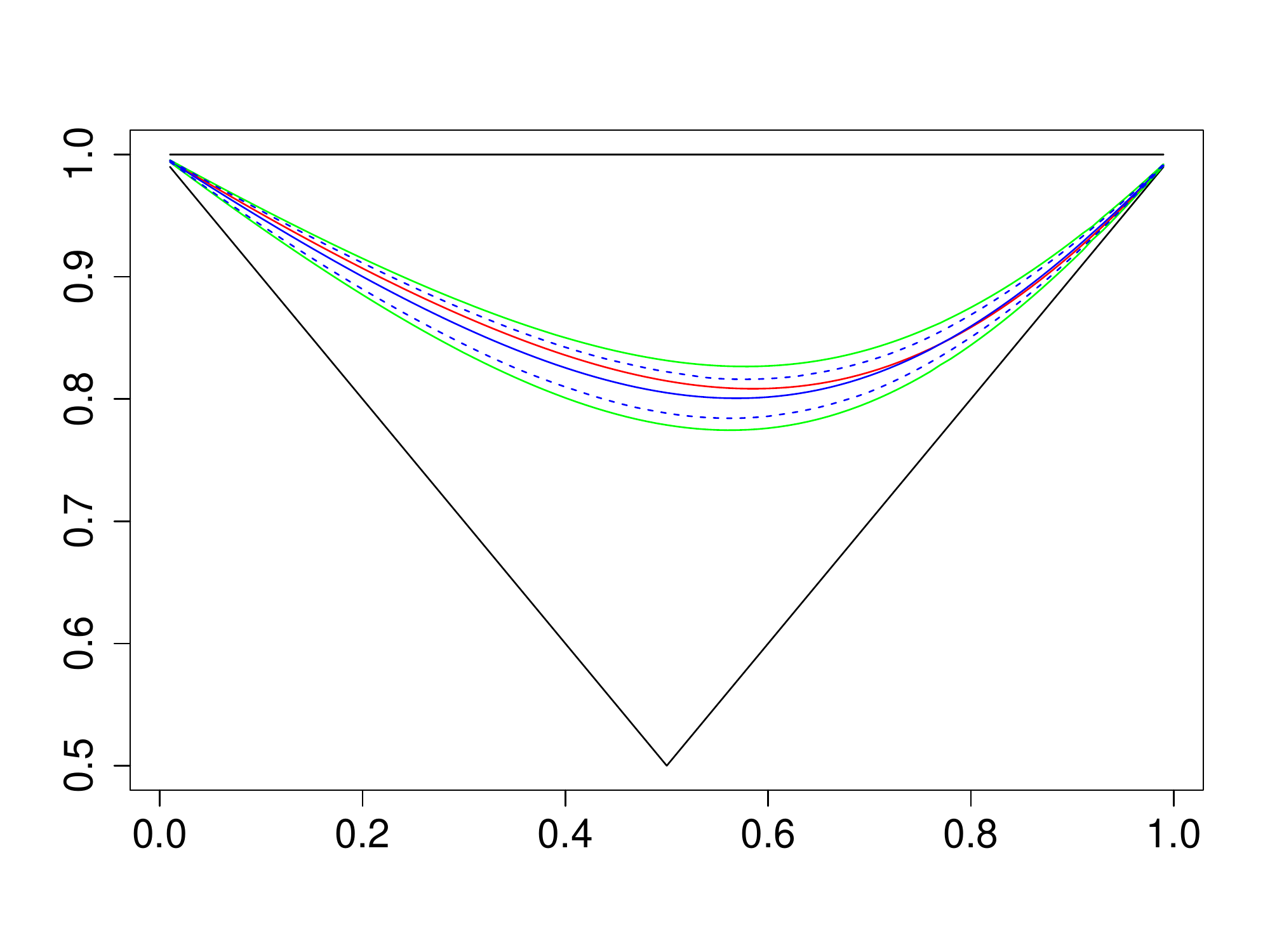}
\end{tabular}
\caption{Left column: histogram of approximate~$Y$, the true spectral distribution (solid red), the fitted distribution (blue), non-parametric density estimate used to compute divergence (dashed red).
Right column: $A_{\textup{true}}(z)$ (red), fitted $A(z)$ (solid blue) with its bootstrap bounds (dashed blue), and the robust square-root bounds (green). The rows correspond to experiments $\#$1--4 described in Table \ref{table:exp}, respectively.}
\label{fig:plots}
\end{center}
\end{figure}

Let us briefly discuss the choice of a dominating measure~$\mu$. We use  $\mu=\p$, i.e., the classical R\'enyi divergence as defined in~\eqref{eq:divergence}, whenever possible, that is, whenever~$\delta_{\textup{true}}$ is finite. %, and so for $\delta\geq \delta_{\textup{true}}$ the true model must be in the neighbourhood of measures over which we optimize. 
This is the case when fitting extremal-$t$ in experiments $\#1$ and $\#3$ and asymmetric logistic in $\#4$, but not in experiment $\#2$. Even though the true symmetric logistic density with no point masses is absolutely continuous with respect to the fitted H\"usler--Reiss density, the R\'enyi divergence of the former from the latter is infinite, because the H\"usler--Reiss density decays faster than any power at~0. Therefore, we take $\mu=$~Leb as a dominating measure in this case. Notice however, that the Lebesgue dominating measure does not allow for point masses which is desired in the other experiments.

%% In this regard it must be observed that the exact robust bounds are not necessarily Pickand's functions themselves, because optimization is done for each $z$ separately, and so different spectral distributions are used. 

Our experiments show that the easily computable robust square-root bounds can 
be applied effectively to measure uncertainty related to misspecified dependence 
structures in multivariate extremes. These readily available bounds are often exact, or very close to the exact robust bounds, see also Section~\ref{sec:illustration}. Thus the more challenging computation of the exact bounds is usually not required.
Let us note that estimation of $\delta$ can be subtle, but it can be improved by an adequate choice of the dominating measure $\mu$.  
Another important problem concerns reliable estimation of $\delta$ when data is coming from a pre-limit distribution. We leave these statistical questions for future research.

\section{Robust bounds on the Value-at-Risk of a portfolio}\label{sec:var}
In recent years diversification effects in heavy-tailed portfolios received considerable attention; see~\cite{mainik_embrechts,mainik_var,zhou}. Suppose that $\bs Z$ is a $d$-dimensional vector of dependent risk factors, and consider the
portfolio $P = \sum_{i=1}^d w_i Z_i$, where $w_1,\dots, w_d\geq 0$ are some non-negative weights, not all being~0; one may assume that $\sum_i w_i=1$ but this is not required. In order to have comparable risks one assumes that $\bs Z$ is multivariate regularly varying with some index $\a>0$ and non-degenerate exponent measure~$\nu$, so that~\eqref{eq:Fequiv} holds true.

Let $\VaR_i(p)$ be the \emph{Value-at-Risk} of the $i$th component, i.e., it satisfies $\overline F_i(\VaR_i(p))=p$, where $p>0$ is a number close to~0. 
It follows that $\VaR_i(p)$ is regularly varying at $0$ with index $-1/\a$, and moreover from~\eqref{eq:Fequiv} we find 
\begin{equation}\label{eq:vari}\frac{\VaR_i(p)}{\VaR_1(p)}\to m_i^{1/\a}\qquad\text{ as }p\downarrow 0.\end{equation}
In the following we consider the Value-at-Risk $\VaR_P(p)$ of the portfolio~$P$ and provide the corresponding asymptotic robust bounds.

As discussed in Section~\ref{sec:reg_var}, it is a common procedure to first estimate the marginal tails and then to address tail dependence, comparable to the copula concept in multivariate modeling. In this work we focus on the latter, more difficult task, and so we assume that the marginal survival functions $\overline F_i$ are correctly specified. 
%For simplicity we also assume that $F_i$ are continuous. 
Transforming the marginals to unit Pareto
\[\widehat Z_i=\frac{1}{\overline F_i(Z_i)},\]
 we obtain normalized multivariate regularly varying $\bs{\widehat Z}$, to which we associate $\widehat{\bs Y}\in\mathbb S^{d-1}$ having the corresponding spectral distribution.

According to~\cite[Thm.\ 3.1]{zhou} the Value-at-Risk $\VaR_P(p)$ of the portfolio~$P$  satisfies
\begin{equation}\label{eq:varP}\left(\frac{\VaR_P(p)}{\VaR_1(p)}\right)^\a\to d\e\left(\sum_{i=1}^d w_i(m_i \widehat Y_i)^{1/\a}\right)^\a\qquad\text{ as }p\downarrow 0.\end{equation}
That is, the Value-at-Risk of the portfolio is asymptotically equivalent to the Value-of-Risk of every individual risk factor up to a diversification constant, which is easily identified from~\eqref{eq:varP} and~\eqref{eq:vari}.
In particular, letting
\begin{equation}\label{eq:X} X = \left(\sum_{i=1}^d w_i(m_i\widehat Y_i)^{1/\a} \right)^\a\end{equation}
we have the approximation for small $p>0$
\[\VaR_P(p)\approx\VaR_1(p)(d\e X)^{1/\a}.\]

Suppose now that our model $\p$ for the extremal dependence between the risk factors, i.e., for the spectral distribution~$\widehat{\bs Y}$, is subject to model uncertainty.
A prominent problem in the financial literature on model uncertainty is to obtain worst case bounds on the
Value-at-Risk of a portfolio \citep[cf.,][]{Embrechts2013,Embrechts2015}. In this regard we may directly apply our results from Sections~\ref{sec:convex} and~\ref{sec:rob_bounds} by considering the optimization problem \eqref{eq:opt} with $X$ given in~\eqref{eq:X} and the moment constraints $\e \widehat Y_i=1/d$. Let us note again that there are essentially $d-1$ constraints since $\widehat Y_d = 1 - \sum_{i=1}^{d-1} \widehat Y_i$.
For fixed uncertainty radius $\delta >0$, Theorem~\ref{thm:opt} yields the desired exact worst case bounds on $\VaR_P(p)$ which are found numerically by solving a system of $d+1$ non-linear equations.
In higher dimension, solving these equations might be computationally challenging.
On the other hand, the upper and lower square-root bounds in Theorem \ref{thm:main} and Corollary \ref{cor:main} coincide with the exact bounds for $\delta < \delta^*$ and $\delta< \delta_*$, respectively, and 
are otherwise very good approximations. Moreover, they can be easily computed even in higher dimensions. Indeed we only need to evaluate the covariance matrix~$\Sigma_\mu(X,\widehat{\bs Y})$ with respect to the chosen dominating measure~$\mu$. In the default case $\mu=\p$ this can be done, for instance, by Monte Carlo methods based on independent samples from~$\widehat{\bs Y}$. An algorithm for exact and efficient simulation of $\widehat{\bs Y}$ can be found in~\cite{dom2016}.

Differently to~\cite{zhou} and the above discussion, in~\cite{mainik_embrechts} the asymptotic relation between $\VaR_P(p)$ and $\VaR_1(p)$
is expressed using the spectral distribution~$\bs Y$ of the original non-standardized~$\bs Z$.
This approach avoids separating the problem into marginal tail estimation and estimation of the tail dependence structure, which may be beneficial in some situations. It does not, however, allow to use standard models for the spectral measure.
Moreover, in this setting the marginal $\VaR_i(p)$ are affected by a change of the distribution of~$\bs Y$, and the asymptotic expression for the ratio:
%, nor to use the robust bounds presented in this paper. To be more explicit, let $\bs Y\in\mathbb S^{d-1}$ have the spectral distribution corresponding to~$\bs Z$. 
\[\left(\frac{\VaR_P(p)}{\VaR_1(p)}\right)^\a\to \e\left(\sum_{i=1}^d w_i Y_i\right)^\a/\, \, \e Y_1^\a\qquad\text{ as }p\downarrow 0,\]
%where $\bs Y$ is not required to satisfy any  moment constraints. It is now clear that  optimization of the right hand side 
see~\cite[Cor.\ 2.3]{mainik_embrechts},
does not fit into our framework since it is given by a ratio of expectations. A possible way around this problem is to express $\VaR_P(p)$ using the slowly varying function corresponding to the scaling sequence $a_t$ in~\eqref{eq:reg_var}.
%In spite of no moment constraints, optimization of the right hand side is a hard problem, because the objective function is not linear in the $\mu$-density~$L$. 
%In particular, our approach presented in Section~\ref{sec:convex} and Section~\ref{sec:rob_bounds} can not be used.

\section*{Acknowledgments}
We are grateful to Enkelejd Hashorva for useful remarks, and to two anonymous referees for in-depth reading of this manuscript and for various comments that substantially improved the paper. Financial support by the Swiss National Science Foundation Project 161297 (first author) and the T.N.~Thiele Center at Aarhus University (second author) is gratefully acknowledged.

\appendix
\section*{Appendix}
\renewcommand{\thesection}{A} 
\subsection{Some parametric families of spectral distributions}\label{appendix_distributions}
In the following we provide several commonly used parametric models
for the spectral distribution~$H$ of $(Y_1,Y_2)\in\mathbb S^1$ in the bivariate case under $L_1$-norm. Without loss of generality, we restrict our attention to the first component~$Y=Y_1$, so that $H$ is a probability measure on $[0,1]$ equipped with its Borel $\sigma$-algebra. The following formulas are known but not readily available in the literature, and so we present them here for completeness.

\subsubsection{ H\"usler--Reiss}\label{sec:HR}
If the max-stable distribution is a bivariate H\"usler--Reiss distribution
with dependence parameter $\lambda \in (0,\infty)$, then the density of the
corresponding spectral distribution is
$$h_\lambda(\omega) = \frac{1}{\omega^2( 1-\omega)4\lambda } \frac{1}{\sqrt{2\pi}} \exp\left\{-\frac12\left(\lambda + \frac{\log\frac{1-\omega}{w}}{2\lambda}\right)^2\right\}, \quad \omega \in[0,1].$$
It is easy to check that this distribution is symmetric around $1/2$ and that $\e Y = 1/2$.
Moreover, $A(1/2) = \Phi(\lambda)$, where $\Phi$ is
the standard normal distribution function.

\subsubsection{Asymmetric logistic}\label{sec:AL}
If the max-stable distribution is an asymmetric logistic distribution
with dependence parameter $a \in (0,1)$ and asymmetry parameters $b_1,b_2\in[0,1]$, then the
corresponding spectral distribution has
point masses $\p(Y=0) = (1-b_2)/2$ and  $\p(Y=1) = (1-b_1)/2$,
and the density is
\begin{equation}\label{eq:AL_dens}
h_{a,b_1,b_2}(\omega) = \frac{1-a}{2a} \frac{(b_1b_2)^{1/a}}{ (\omega(1-\omega))^{1 + 1/a }} \left\{ \left(\frac{b_1}{\omega}\right)^{1/a} + \left(\frac{b_2}{1-\omega}\right)^{1/a}\right\}^{a-2}, \quad \omega \in(0,1).
\end{equation}

\subsubsection{Extremal-$t$}\label{sec:ET}
If the max-stable distribution is an extremal-$t$ distribution with
parameters $\rho\in[-1,1]$ and $a >0$, then the distribution  of the
corresponding spectral distribution has
point masses 
$$\p(Y=0) = \p(Y=1) = 1 - F_{a+1}\left\{ \rho \sqrt{\frac{a+1}{1-\rho^2}}\right\},$$ 
and the density for $\omega \in(0,1)$ is 
$$h_{\rho,a}(\omega) = \frac{(1-\rho^2)^{\frac{a+1}{2}}\Gamma(\frac{a +2}{2})}{2a \sqrt{\pi} \Gamma(\frac{a +1}{2})} (\omega(1-\omega))^{1/a-1} \left\{ \omega^{2/a} - 2\rho(\omega(1-\omega))^{1/a} + (1-\omega)^{2/a}\right\}^{-\frac{a+2}{2}}.$$
Here, $F_a$ is the $t$-distribution function with $a > 0$ degrees of freedom, i.e.,
$$F_a(x) = \frac{\Gamma(\frac{a +1}{2})}{\sqrt{a\pi}\Gamma(\frac{a }{2})} \int_{-\infty}^x \left( 1 + \frac{t^2}{a}\right)^{-\frac{a +1}{2}} \mathrm d t,  \quad x\in\mathbb R.$$

\subsection{Degenerate optimizers for the Pickands' function}\label{sec:delta_starstar}
In this section we study the degenerate case (ii) of Theorem~\ref{thm:opt} for the Pickands' dependence function~$A(z)$ defined in~\eqref{eq:pickands}, see also Section~\ref{sec:pickands_bounds}. That is, we assume that 
\begin{equation}\label{eq:Xpickands}
X=X(z)=2\{(1-z)Y\vee z(1-Y)\},\qquad Y\in[0,1]
\end{equation}
for some fixed $z\in[0,1]$, and that the constraint is~$\e Y=1/2$.
 Our aim is to identify the corresponding optimal values and the thresholds $\delta^{**}$ and $\delta_{**}$, see Remark~\ref{rem:case_ii}. The following result shows that the degenerate case corresponds to the trivial bounds $z\vee(1-z)\leq A(z)\leq 1$ as expected, but a certain assumption on the $\mu$-support of $Y$ is necessary.
\begin{lemma}\label{lem:mass}
Consider~\eqref{eq:Xpickands} and assume
that $\mu$-support of $Y$ contains~$0,z,1$. Then the case (ii) of Theorem~\ref{thm:opt} occurs if and only if there exists a Radon--Nikodym derivative $L^*$ (with respect to $\mu$) such that
\[\Em(L^*-L)^2\leq \delta,\qquad\p^*(Y=0)=\p^*(Y=1)=1/2,\] 
in which case $\e^* X=1$.

In case of the minimization problem the corresponding requirement on a Radon--Nikodym derivative $L_*$ is
\[\Em(L_*-L)^2\leq \delta,\qquad\begin{cases}\p_*(Y\geq z)=1, &z<1/2,\\
\p_*(Y=1/2)=1, &z=1/2,\\
\p_*(Y\leq z)=1, &z>1/2,
\end{cases}\qquad \e_* Y=1/2,\]
in which case $\e_* X=z\vee(1-z)$. 
\end{lemma}
\begin{proof}
%\begin{figure}[h!]
%\includegraphics[width=0.35\textwidth]{Lem_mass.pdf}
%\caption{$X$ as a function of $Y$ and some possible $X+cY$ for $z=1/3$.}
%\label{fig:lem_mass}
%\end{figure}
Observe that the maximum of $X+c Y$ is obtained for $Y=1$ or $Y=0$ or both (draw a picture). %see Figure~\ref{fig:lem_mass}. 
Since $\e^* Y=1/2$ we must have $\p^*(Y=0)=\p^*(Y=1)=1/2$, which yields the result. 

The minimum of $X+cY$ is obtained either at $Y\leq z$ or at $Y\geq z$ or at the single points $0,z,1$ ($z$ is the bending point). Again the constraint $\e_*Y=1/2$ leads to the result. The corresponding optimal value is $2(1-z)\e_* Y=1-z$ when $z\leq 1/2$, and it is $z$ when $z>1/2$.
\end{proof} 

For the maximization problem, in the case of $\mu=\Leb$, it is impossible to have $\p^*(Y=0),\p^*(Y=0)>0$ and so according to Lemma~\ref{lem:mass} there cannot be a degenerate maximizer for any $\delta$, i.e.~$\delta^{**}=\infty$.
In the case of $\mu=\p$ we have the following:
\begin{equation}\label{eq:mass}\delta^{**}=\frac{1}{4p_0}+\frac{1}{4p_1}-1,\end{equation}
where $p_i=\p(Y=i)$, and in particular $p_0$ and $p_1$ must be positive to have $\delta^{**}<\infty$. 
This follows from Lemma~\ref{lem:mass} and the following arguments. Note that
\[\e {L^*}^2= \e({L^*}^2|Y=0)p_0+\e({L^*}^2|Y=1)p_1\geq l_0^2p_0+l_1^2p_1,\]
where $l_i=\e(L^*|Y=i)$ and so $l_0p_0=l_1p_1=1/2$. Hence, $L^*=l_0\1{Y=0}+l_1\1{Y=1}$ guarantees the minimal value for $\e {L^*}^2$ among the allowed ones for any fixed $\delta$. Therefore, a sufficient and necessary condition for existence of a degenerate maximizer is $l_0^2p_0+l_1^2p_1-1\leq\delta$, which readily yields~\eqref{eq:mass}.

For the minimization problem, the case of $z=1/2$ is easy, $\delta_{**}=\infty$ for $\mu=\Leb$, and $\delta_{**}=1/\p(Y=1/2)-1$ for $\mu=\p$.
For $z\neq 1/2$ the value of $\delta_{**}$ depends on the distribution of $Y$ on $Y\geq z$ if~$z<1/2$ (on $Y\leq z$ if $z>1/2$). 
More precisely, we need to identify a Radon--Nikodym derivative $L_{**}$ which assigns no mass to $Y<z$, satisfies the constraints $\Em(L_{**})=1,\Em(L_{**}Y)=1/2$ and minimizes $\Em ({L_{**}-L})^2$.
This optimization problem is solved by $L_{**}=(b+cY+L)_+\1{Y\geq z}$ for $b,c\in\R$ such that the constraints hold. Finally, the minimal value $\Em ({L_{**}-L})^2$ is our~$\delta_{**}$.

\subsection{Optimization for R\'enyi and Kullback--Leibler divergences}\label{appendix_KL}
For completeness, we consider our optimization problem~\eqref{eq:opt} for some other popular divergences: R\'enyi divergence of order~$\eta>1$ given by 
\[\widehat D_\eta(\p',\p)=\frac{1}{\eta-1}\log\e L'^\eta,\] and Kullback--Leibler divergence given by
\[\widehat D_1(\p',\p)=\e (L'\log L'),\]
where it is assumed that $\p'\ll\p$ with $L'=\D \p'/\D\p$ and that the dominating measure~$\mu$ coincides with~$\p$.
An easy adaptation of the proof of Theorem~\ref{thm:opt} shows that a maximizer $\p^*$ of \[\sup_{\p'}\{\e' X:\widehat D_\eta(\p',\p)\leq \delta,\e' \bs Y=\e\bs Y\}\] must have a Radon--Nikodym derivative~$L^*\geq 0$ which satisfies $\e L^*=1,\e (L^*\bs Y)=\e\bs Y$ and one of then following:
\begin{itemize}
\item [(i)] $\widehat D_\eta(\p^*,\p)=\delta$ and there exist $a>0,b,c_i\in\R$ such that
\begin{align}
\label{eq:power} L^*&=(aX+b+\bs c^\top\bs Y)_+^{1/(\eta-1)}\quad \text{a.s.}&\text{when }\eta>1,\\
\label{eq:KL} L^*&=\exp(aX+b+\bs c^\top\bs Y)\quad\text{a.s.}&\text{when }\eta=1,
\end{align}
\item [(ii)]there exist $c_i$ such that the distribution of $X+\bs c^\top\bs Y$ has a positive mass at its right end, $L^*=0$ everywhere else a.s., and the constraint $\widehat D_\eta(\p^*,\p)\leq \delta$ holds.
\end{itemize}
Conversely, any such $L^*$ corresponds to a maximizer~$\p^*$.
In the case $\eta>1$ we assume that $\e |X|^{\eta/(\eta-1)},\e |Y|^{\eta/(\eta-1)}<\infty$, and in the case $\eta=1$ we assume that 
$\e(|X|L'),\e(|Y_i|L')<\infty$ for all $L'$ satisfy $\e (L'\log L')\leq \delta$. 
Furthermore, regardless of these assumptions, if there exists~$L^*$ as above and such that $\e(|X|L^*),\e(|Y_i|L^*)<\infty$ then it must be a maximizer, which can be seen by considering an appropriate convex subset of $L'$ in the proof of Theorem~\ref{thm:opt}.

%\begin{remark}\label{rem:assumptions}
%Alternatively to the above assumptions on $X$ and $\bs Y$, we may assume that $\e Y_i<\infty$ and $X$ is non-negative. This follows by considering a convex subset of $L'$ satisfying $\e (L'X),\e (L'Y_i),\widehat D_\eta(\p',\p)<\infty$. It is then left to show that existence of the corresponding maximizer $L^*$ implies that there is no Radon--Nikodym derivative $\hat L$ which satisfies the constraints but leads to an infinite optimal value $\e(\hat L X)=\infty$. Suppose there is such $\hat L$ then we may construct a sequence of derivatives $\hat L_n\rightarrow L$ such that $\hat L_n$ also satisfies the constraints and $\e (\hat L_n X)<\infty$ but unbounded, and so $L^*$ can not be a maximizer of the problem under additional assumption of finiteness. The sequence $\hat L_n$ can be obtained by making $L$ constant on the event $\{L>n\}$ 
%
%What about the constraints on~$Y$?
%\end{remark}

Note that taking $\eta=2$ we retrieve the result of Theorem~\ref{thm:opt} for $\mu=\p$. In the case $d=1$ (no moment constraints) the expression for $L^*$ in (i) appears in e.g.~\cite{blanchet2016extreme}. Furthermore,~\cite{breuer2013measuring} considers more general divergences but the results are less explicit.
Finally, we elaborate on the case of Kullback--Leibler divergence extending the result of~\cite{ahmadi} by introducing moment constraints.

%\begin{remark}
%Let us comment on the form of Theorem~\ref{thm:opt} for Kullback-Leibler divergence, i.e.\ when optimization is done over a set $\mathcal P_1(\delta)=\{L\geq 0:\e L=1,\e(L\log L)\leq \delta\}$. A quick inspection of the proof shows that either (i)
%\[L^*=\exp(\lambda X+\mu+\bs\eta^\top\bs Y)\]
%for some $\lambda>0,\mu,\eta_i\in\R$, where $\e L^*=1,\e (L^*\log L^*)=\delta,\e(L^*\bs Y)=\bs y$, or (ii) the distribution of $X+\bs\eta^\top\bs Y$ has a positive mass at its right end, $L^*=0$ everywhere else, and $\e L^*=1,\e (L^*\log L^*)\leq \delta,\e(L^*\bs Y)=\bs y$.
%Technically speaking, our proof requires an assumption that $\e(|X|L)<\infty$ for all $L\in\mathcal P_1(\delta)$.
%\end{remark}

\begin{prop}
Assume that $X,Y_i\geq 0$ are positive random variables with finite expectation, and $G(a,\bs c)=\e e^{a X+\sum_i c_i Y_i}$ is finite on some domain $\mathcal D\subset (0,\infty)\times\R^{d}$ with non-empty interior. Suppose there exist $(a,\bs c)\in\mathcal D$ such that 
%Define a set of Radon--Nikodym derivatives for $\delta>0$ by
%\[\mathcal P^{KL}_{\delta}=\{L\geq 0:\e L=1,\e (L\log L)\leq \delta\}\]
%and consider the following optimization problem:
%\[V^{KL}(\delta)=\sup_{L\in\mathcal P^{KL}_\delta}\{\e_L(Z):\e_L(Z_i)=\e Z_i\quad\forall i=1,\ldots,n\}.\]
\[\frac{G_i(a,\bs c)}{G(a,\bs c)}=\e Y_i,\quad a\frac{G_0(a,\bs c)}{G(a,\bs c)}+\sum_{i=1}^d c_i\e Y_i-\log G(a,\bs c)=\delta,\]
where $G_i(\cdot)$ is a derivative with respect to the $i$th variable (pointing inside the domain if on the boundary).
Then \[V^{KL}(\delta)=\sup_{\p'}\{\e' X:\widehat D_1(\p',\p)\leq \delta,\e'\bs Y=\e \bs Y\}=\frac{G_0(a,\bs c)}{G(a,\bs c)},\]
which corresponds to the exponential change of measure $L^*=e^{a X+\sum_i c_iY_i}/G(a,\bs c)$.
\end{prop}
\begin{proof}
According to~\eqref{eq:KL} we consider
\[L^*=\exp(a X+b+\sum_ic_iY_i)=:\exp(U),\quad a>0,b,c_i\in\R\]
together with the constraints: $\e e^U=1,\e (Ue^U)=\delta,\e(Y_ie^U)=\e Y_i$. 
We may rewrite these using the moment generating function $G$:
\begin{align*}
G(a,\bs c)&=e^{-b},\\
a G_0(a,\bs c)+b G(a,\bs c)+\sum_{i} c_iG_i(a,\bs c)&=\delta e^{-b},\\
e^b G_i(a,\bs c)&=\e Y_i,\\
V^{KL}(\delta)&=e^b G_0(a,\bs c).
\end{align*}
The equations in the statement are now immediate. Finally, we note that $\e(XL^*)$ and $\e(Y_iL^*)$ are finite which completes the proof.
\end{proof}
%There is an interesting representation of $V^{KL}(\delta)$ in~\cite{ahmadi} (when $n=0$):
%\[V^{KL}(\delta)=\inf_{\lambda>0}\frac{\log G(\lambda)+\delta}{\lambda}.\]
%Indeed, take the derivative of the function below the $\inf$ and equate it to~0 to get
%\[G'(\lambda)/G(\lambda)=(\log G(\lambda)+\delta)/\lambda,\]
%The equivalence is now obvious.
%
%see also~\cite{ahmadi} (\url{https://en.wikipedia.org/wiki/Entropic_value_at_risk}); here however we add additional constraints on the expectation of $Z_i$.

%\bibliography{model_risk}
%\bibliographystyle{abbrv}

\end{document}